\documentclass[11pt,oneside]{amsart}
\usepackage[utf8]{inputenc}
\usepackage{mathrsfs}

\usepackage{setspace}
\usepackage{amsfonts}
\usepackage{amsmath}
\usepackage{amssymb}
\usepackage{amsthm}
\usepackage{stmaryrd}
\usepackage{breqn}
\usepackage{graphicx,xcolor} % Required for inserting images
\usepackage[left=30mm,right=30mm,top=30mm,bottom=30mm]{geometry}
\usepackage{hyperref}
\usepackage{ytableau}
\usepackage{physics}
\usepackage{esint}
\usepackage{todonotes} 
\usepackage{breqn}
\usepackage{mathtools}
\usepackage{tikz}
\usepackage{subcaption}
\usepackage{enumitem}

\usepackage{graphicx}

\newcommand{\case}[1]{\smallskip\noindent\textit{Case #1.}}
\theoremstyle{definition}
\newtheorem{theorem}{Theorem}[section]
\newtheorem{proposition}[theorem]{Proposition}

\newtheorem{example}[theorem]{Example}
\newtheorem{lemma}[theorem]{Lemma}

\theoremstyle{remark}
\newtheorem{remark}[theorem]{Remark}

\numberwithin{equation}{section}

\hypersetup{colorlinks,breaklinks,
            linkcolor = purple,
            citecolor = teal,
            urlcolor = purple,
            }
\definecolor{darkblue}{rgb}{0.1,0.1,0.7}
\definecolor{darkred}{rgb}{0.5,0.1,0.1}
\definecolor{darkgreen}{rgb}{0.0,0.42,0.06}

\title{Multicritical Scaling Limit of Shifted Schur Measure}

\begin{document}

\begin{abstract}
We investigate the multicritical scaling limit of the shifted Schur measures. 
Under an appropriate scaling limit and specific conditions on the continuous parameters, we explicitly determine the limit shape of strict partitions distributed according to the shifted Schur measure. 
We then show that, under a multicritical condition, the edge scaling limit of the correlation function converges to a determinant of the higher-order Airy kernel. This rigorously demonstrates a transition from a Pfaffian point process to a determinantal distribution in the scaling limit.
\end{abstract}

\author{Haruna Aida}
\address{Department of Mathematics, Hokkaido University, Sapporo, Japan}

\author{Taro Kimura}
\address{Université Bourgogne Europe, CNRS, IMB UMR 5584, Dijon, France}

\maketitle

\setcounter{tocdepth}{1}
\tableofcontents

%\newpage

\section{Introduction}
A \emph{random partition} is a random variable that takes a value in the set of partitions. 
In the large-scale limit, the profile of partitions, distributed under an appropriate measure, converges to a continuous curve known as the \emph{limit shape}. 
In this context, fluctuation in the vicinity of the profile edge is particularly interesting, since it typically shows a universal structure observed in various probabilistic and statistical models.
The limit to focus on the edge fluctuation is called the \emph{edge scaling limit}.

The limit shape was first discovered for the Plancherel measure by Logan and Shepp \cite{MR1417317} and Vershik and Kerov \cite{MR480398}. 
Via Kerov’s transition measures \cite{MR1251166}, this limit shape matches Wigner’s semicircle of Gaussian random matrices. 
In fact, this convergence can be viewed as the free central limit theorem in free probability \cite{MR1644993}. 
Moreover, as shown by Baik, Deift, and Johansson~\cite{MR1682248}, Borodin, Okounkov, and Olshanski~\cite{MR1758751} and Johansson~\cite{MR1826414}, the edge scaling limit of the random permutations and the Plancherel measure yields the \emph{Tracy-Widom distribution}, originally introduced to describe the fluctuation of the largest eigenvalue of the Gaussian Unitary Ensemble \cite{MR1257246}. 
As these examples suggest, different probability models often share the universal behaviour in the edge scaling limit.

We may regard the Plancherel measure as a special case of the \emph{Schur measures} introduced by Okounkov~\cite{MR1856553}. 
The Schur measure assigns to each partition $\lambda$ a probability proportional to the product of two Schur functions:
\begin{equation*}
    \mathbb{P}_{\text{S}}(\lambda) = \frac{1}{Z_{\text{S}}} s_\lambda(\mathbf{x}) s_\lambda(\mathbf{y}),
\end{equation*}
where $\mathbf{x}$ and $\mathbf{y}$ are sets of parameters, and $Z_{\text{S}}$ is a normalisation constant depending on $\mathbf{x}$ and $\mathbf{y}$. This algebraic definition naturally yields a \emph{determinantal point process}, whose correlation functions are given by determinants constructed from a correlation kernel, and satisfy the KP (or 2D Toda lattice) hierarchy \cite{MR1856553}. 
In the context of the boson-fermion correspondence of type $A_{\infty}$, the Schur measure is described via a \emph{coherent state} (a simultaneous eigenstate of annihilation operators) of \emph{charged} free fermions. The variables $\mathbf{x}$ and $\mathbf{y}$ (or their power-sum counterparts, called the Miwa variables) then represent the eigenvalues of this state. To determine the limit shape and the edge scaling limit explicitly, we need to evaluate these eigenvalues for a specific choice of parameters.

Betea, Bouttier, and Walsh \cite{MR4693927} (announced in~\cite{Betea:2020}) studied a specialisation called the \emph{multicritical Schur measures}, finding both the limit shape and the edge scaling limit. 
The edge scaling limit corresponds to the higher-order Tracy-Widom distribution \cite{PhysRevLett.121.030603}, which can be obtained by replacing the Airy function in the definition of Tracy--Widom distribution with the {\it higher Airy function}.
\footnote{Depending on the context, several different definitions of the higher Airy function are known (see e.g., \cite{MR1171758}). We follow the convention in \cite{MR4700870}.}
These results were also obtained independently by the second-named author and Zahabi \cite{Kimura:2020sud,MR4700870}; see also \cite{Kimura:2021lrc}. 
Whereas Betea, Bouttier, and Walsh directly computed the correlation kernel in the multicritical limit, the second-named author and Zahabi demonstrated that the auxiliary functions converge to higher Airy functions, which in turn implies that the correlation kernel converges to the higher Airy kernel. 

\subsubsection*{Summary of results}
In this paper, we investigate the multicritical scaling limit for the \emph{shifted Schur measures} \cite{MR2060026}. They are rooted in the boson-fermion correspondence of type $B_\infty$ for \emph{neutral} (or \emph{Majorana}) fermions \cite{MR723457}. Algebraically, for a \emph{strict} partition $\lambda$, the shifted Schur measure is given by
\begin{equation*}
    \mathbb{P}_{\text{SS}}(\lambda; t) = \frac{1}{Z} P_\lambda(t) Q_\lambda(t),
\end{equation*}
where $P_\lambda(t),Q_\lambda(t)$ are Schur's $P,Q$-functions, respectively, depending on $t=(t_{1},t_{3},\dots)$ a set of infinite parameters, and $Z$ is a normalisation constant. See \autoref{sec:SS_measure} for details. This change in the underlying fermionic structure—from charged to neutral—is reflected in the stochastic properties of the measure: it gives rise to a \emph{Pfaffian point process}, and its correlation functions satisfy the BKP hierarchy rather than the KP hierarchy \cite{MR2139724, WANG2019447}.
Our first main result describes the limit shape of strict partitions under the shifted Schur measure:

\begin{theorem}[Theorem~\ref{thm:limit_shape}]
    Subject to specific conditions on the parameter $t=(t_{1},t_{3},\dots)$, and under an appropriate scaling limit, the limit shape of a partition distributed according to the shifted Schur measure is given by
    \begin{equation*}
    x+\frac{2}{\pi} \int_0^\pi \max\{D(\theta)-x,0\}{\rm d}\theta \quad (x\ge 0),
    \end{equation*}
    where $D(\theta) = 4 \sum_{n: \text{odd}}nt_n \cos{n\theta}$.
\end{theorem}

%Our conditions on the parameter $t$ are somewhat broader than those required for the standard Schur measure as shown in \cite{MR4693927}. 
In order to discuss the limit shape itself, the multicriticality condition is not yet necessary. 
Although it may not be immediately obvious from the formula, this limit shape is non-differentiable at the point $x=4\sum_{n:\text{odd}}nt_{n}$, which is identified as the profile edge. We provide visual examples of these shapes in \autoref{fig:Density_Shape}.

We now shift our focus to the fluctuations near the edge, specifically in the specific scaling limit.
More generally, we study the edge scaling limit of the $N$-point correlation function:
\begin{equation*}
    \rho_{\rm SS}(a_1,\ldots,a_N;t) = \sum_{\lambda \supset \{a_1,\ldots,a_N\}} \mathbb{P}_{\rm SS}(\lambda;t).
\end{equation*}
To investigate the edge scaling behaviour, we prescribe the $p$-multicritical condition on the parameter $t$ for $p\in 2 \mathbb{Z}_{>0}$. 
Our second main result is stated as follows:

\begin{theorem}[Theorem~\ref{thm:kernel_limit}]
    Suppose the parameter $t=(t_{1},t_{3},\dots)$ satisfies specific conditions including the $p$-multicritical condition. 
    Then, under an appropriate scaling, the edge scaling limit of the $N$-point correlation function is given by
    \begin{equation*}
        \det_{1\le i,j \le N} K_{p\text{-Airy}}(a_i,a_j),
    \end{equation*}
    where $K_{p\text{-Airy}}$ denotes the $p$-Airy kernel (see \autoref{app:higher_Airy}).
\end{theorem}

The case $p=2$ was obtained by Matsumoto \cite{MR2139724}. Remarkably, although the shifted Schur measure itself is a Pfaffian point process, its edge scaling limit converges to the \emph{determinant} of the $p$-Airy kernel. This transition arises because the $N$-point correlation function can be expressed as the Pfaffian of a $2N\times 2N$ matrix; as we take the scaling limit, the diagonal blocks of this matrix vanish. Consequently, the Pfaffian reduces to a determinant via the identity:
\begin{equation*}
    \mathrm{Pf}\begin{pmatrix}
    O & A\\ -A^{\top} & O
\end{pmatrix} = (-1)^{\binom{N}{2}} \det A.
\end{equation*}
The proof of this theorem provides a rigorous foundation for the approach suggested in \cite{Kimura:2020sud,MR4700870}.

\subsubsection*{Organisation of the paper}
The remainder of this paper is organized as follows. 
\autoref{sec:BF_corresp} briefly reviews the theory of neutral fermions and the associated boson-fermion correspondence, following \cite{WANG2019447}. 
In \autoref{sec:SS_measure}, we formulate the shifted Schur measures in terms of neutral fermions, showing that a correlation function are given by a Pfaffian, and deriving a double contour integral representation for the correlation kernel. 
\autoref{sec:lim_shape} is devoted to the limit shape analysis of the shifted Young diagrams, where we prove our first main theorem (Theorem \ref{thm:limit_shape}). 
In \autoref{sec:edge_lim}, we focus on the multicritical edge scaling limit and prove our second main theorem (Theorem \ref{thm:kernel_limit}), establishing the convergence of the correlation function to the determinant of the $p$-Airy kernel. 
Finally, \autoref{app:higher_Airy} summarizes the fundamental properties of higher Airy functions.

%\subsubsection*{Notation} 
%\rem{If needed}

\subsubsection*{Acknowledgments}
HA would like to express her gratitude to her primary supervisor, Professor Takahiro Hasebe, for his continuous mentorship and warm encouragement. 
She is also thankful to her secondary supervisor, Professor Ryosuke Sato, for taking the time to attend a seminar on this work and offering helpful comments.
The visit to UBE/IMB for the final stage of this project was supported by JSPS KAKENHI Grant Number JP22K03346.
This work was supported by EIPHI Graduate School (No.~ANR-17-EURE-0002) and the Bourgogne-Franche-Comté region.

\section{Neutral Fermions and Their Boson-Fermion Correspondence}
\label{sec:BF_corresp}
To pave the way for the definition of the shifted Schur measure, we first recall the essentials of the theory of neutral fermions (also known as Majorana fermions) and the associated boson-fermion correspondence, following the treatment in \cite{WANG2019447}.

Let $\psi_a,\psi_a^* \, (a \in \mathbb{Z})$ denote the charged fermions. They satisfy the standard anti-commutation relations 
\begin{equation*}
   [\psi_a,\psi_b]_{+}=0, \quad [\psi^{*}_a,\psi^{*}_b]_{+}=0, \quad [\psi_a,\psi^{*}_b]_{+}=\delta_{a+b,0},
\end{equation*} 
where $[A,B]_{+}=AB+BA$ denotes the anti-commutator.
Let $W_A$ be the complex linear space spanned by $\{ \psi_a,\psi_a^*\}_{a \in \mathbb{Z}}$. Then, these relations naturally give rise to the Clifford algebra $Cl(W_A)$ associated with the charged fermions.

The algebra $Cl(W_A)$ admits a canonical representation constructed as follows.
Consider a linear space spanned by the basis vectors $\{e_a\}_{a \in \mathbb{Z}}$, and let $\mathcal{F}_A$ be the Fock space generated by infinite wedge of the form 
\begin{equation}
    \cdots \wedge e_{a_3} \wedge e_{a_2} \wedge e_{a_1}, \quad a_1 > a_2 > a_3 > \cdots, 
    \label{eq:infwedge} 
\end{equation}
subject to the condition that the set of negative integers omitted from $\{a_k\}_{k=1,2,\ldots}$ is finite. 
We define the representation $\varrho_A \colon Cl(W_A) \to \mathrm{End}(\mathcal{F}_A)$ by its action: 
\begin{align*}
    \varrho_A(\psi_a)(\bullet) &= \bullet \wedge e_{-a}, \\
    \varrho_A(\psi^*_a)(\bullet \wedge e_a) &= \bullet.
\end{align*}
Upon interpreting the vector \eqref{eq:infwedge} as a quantum state where the sites $\{a_k\}_{k=1,2,\ldots}$ on a one-dimensional lattice are occupied by charged fermions, $\psi_a$ and $\psi^*_a$ act as creation and annihilation operators, respectively.

The neutral fermions $\phi_a \, (a \in \mathbb{Z})$ are characterised by the following anti-commutation relations: 
\begin{equation}
    [\phi_a,\phi_b]_+ = (-1)^a \delta_{a+b,0}.
    \label{eq:R_neu}
\end{equation}
These operators can be expressed in terms of the charged fermions $\psi_a,\psi^*_a$ via the relation
\begin{equation}
    \phi_a = \frac{\psi_a+(-1)^a\psi^*_{a}}{\sqrt{2}}.
    \label{eq:neutral}
\end{equation}
Letting $W_B$ denote the subspace of $W_A$ spanned by $\phi_a \, (a \in \mathbb{Z})$, we obtain the Clifford algebra $Cl(W_B)$ of neutral fermions as a subalgebra of $Cl(W_A)$.

The term `neutral' comes from the fact that creation and annihilation operators are linearly combined on the right-hand side of \eqref{eq:neutral}. 
In what follows, we shall be concerned with a representation in which $\phi_{a>0}$ and $\phi_{a<0}$ act as creation and annihilation operators, respectively. 
To this end, consider a linear space spanned by $e_a \, (a \ge 0)$, and let $\mathcal{F}_B$ be its exterior algebra.
A basis of $\mathcal{F}_B$ is given by 
\begin{equation}
     e_{a_l} \wedge \cdots \wedge e_{a_1}, \quad a_1 > \cdots > a_l,
     \label{eq:wedge}
\end{equation}
where the case $l=0$ is permitted (corresponding to the scalar $1 \in \mathbb{C}$). 
We define the representation $\varrho_B \colon Cl(W_B) \to \mathrm{End}(\mathcal{F}_B)$ by the action:
\begin{align*}
    \varrho_B(\phi_{a\ge 0})(\bullet) &= \bullet \wedge e_a \\
    \varrho_B(\phi_{-a\le 0)})(\bullet \wedge e_a) &= (-1)^a \bullet.
\end{align*}
Analogous to the charged fermion case, the vector \eqref{eq:wedge} can be interpreted as a quantum state in which the sites indexed by $a_1,\ldots,a_l$ on a one-dimensional lattice are occupied by neutral fermions.

For convenience, we shall henceforth suppress the symbol $\varrho_B$ and set $|0\rangle = 1 \in \mathcal{F}_B$. 
The vector (\ref{eq:wedge}) is then written as 
\begin{equation*}
    \phi_{a_1} \cdots \phi_{a_l} |0\rangle.
\end{equation*}
Representing the state in which no neutral fermions are present, $|0\rangle$ is referred to as the vacuum state.

The linear map 
\begin{equation*}
    \varrho_B(\cdot)(|0\rangle) : Cl(W_B) \longrightarrow \mathcal{F}_B
\end{equation*}
induces a vector space isomorphism 
\begin{equation*}
    \mathcal{F}_B \simeq Cl(W_B) \Big/ Cl(W_B) \Big( \sum_{a<0}\mathbb{C}\phi_a \Big).
\end{equation*}
Dually, we define the space $\mathcal{F}_{B}^{*}$ as the quotient
\begin{equation*}
    \mathcal{F}_{B}^{*} \coloneqq Cl(W_B) \Big/ \Big( \sum_{a>0}\mathbb{C}\phi_a \Big) Cl(W_B)
\end{equation*}
and denote by $\langle 0| \,\in\mathcal{F}_{B}^{*}$ the residue class of $1\in Cl(W_B)$. 
Since the multiplication in $Cl(W_B)$ descends to the map 
\begin{equation*}
    \mathcal{F}_{B}^{*} \times \mathcal{F}_B \to Cl(W_B) \Big/ \Big( \sum_{a>0}\mathbb{C}\phi_a \Big) Cl(W_B) \Big( \sum_{a<0}\mathbb{C}\phi_a \Big) \simeq \mathbb{C}\oplus\mathbb{C}\phi_0,
\end{equation*}
we obtain a bilinear pairing $\mathcal{F}_{B}^{*} \times \mathcal{F}_B \to \mathbb{C}$ by imposing
\begin{equation*}
    \langle 0|1|0 \rangle = 1, \quad \text{and} \quad \langle 0| \phi_0 |0 \rangle = 0.
\end{equation*}
The pairing is referred to as the vacuum expectation value.

Next, we introduce the operator $H_{n}$ defined by
\begin{equation*}
    H_n = \frac{1}{2} \sum_{a\in \mathbb{Z}} (-1)^{a+1} \phi_a \phi_{-a-n}.
\end{equation*}
This is termed the Hamiltonian, a choice motivated in part by the following commutation relation:
\begin{equation}
    [H_n,\phi_a]=
    \begin{cases}
    \phi_{a-n} & \text{for $n$ odd} \\
    0 & \text{for $n$ even}.
    \end{cases}
    \label{eq:R_H_phi}
\end{equation}
By assigning an energy $a$ to each neutral fermion $\phi_a$, the action of $H_n$ on $\phi_a$ can be seen to decrease the energy by $n$.
In this setting, the operators $H_n$ vanish for even $n$, whilst the remaining operators $\{H_n \mid n \in 2\mathbb{Z}+1\}$ generate a Heisenberg algebra satisfying 
\begin{equation}
    [H_m,H_n] = \frac{m}{2}\delta_{n+m,0}.
    \label{eq:Hamilton}
\end{equation}
This correspondence yields a bosonisation of the neutral fermions.

Further to the above, we introduce the formal variables $t=(t_1,t_3,\dots)$ (often referred to as \emph{Miwa variables}) and a complex variable $z \in \mathbb{C}$. 
We define 
\begin{equation*}
    H_+(t) = \sum_{n\ge1, \text{odd}} t_n H_n, \quad \phi(z) = \sum_{a\in \mathbb{Z}} \phi_a z^a, \quad \xi(t,z)=\sum_{n\ge 1, \text{odd}}t_n z^n.
\end{equation*}
These definitions lead to the identity
\begin{equation*}
    e^{H_+(t)}\phi(z)e^{-H_+(t)} = e^{\xi(t,z)}\phi(z).
\end{equation*}

We are now in a position to state the boson-fermion correspondence for neutral fermions \cite{MR723457}. 
First, we partition the Fock space into its even and odd components: 
\begin{equation*}
    \mathcal{F}_B^0 = \text{span} \{\phi_{a_1}\cdots\phi_{a_{2k}} |0\rangle\}, \quad \mathcal{F}_B^1 = \text{span} \{\phi_{a_1}\cdots\phi_{a_{2k-1}} |0\rangle \}.
\end{equation*}
We define the linear maps 
\begin{equation}\label{map_BScor}
\begin{split}
    \sigma^i \colon \mathcal{F}_{B}^i &\to \mathbb{C}[t_1,t_3,\ldots] \\
    |U\rangle &\mapsto \langle i | e^{H_+(t)} | U \rangle,
\end{split}
\end{equation}
where we set $|1\rangle = 2^{\frac{1}{2}}\phi_0|0\rangle$. 
It is worth noting that although $e^{H_{+}(t)}$ is a formal power series, the image $\sigma^{i}(|U\rangle)$ is always a polynomial in $t$; this is guaranteed by the relation \eqref{eq:R_H_phi}.

\begin{theorem}[Boson-fermion correspondence for neutral fermions \cite{{MR723457}}]
\label{BScor}
The homomorphism $\sigma^i$ constitutes an isomorphism of representations between
\begin{equation*}
    (\mathcal{F}_B^i, H_{n\ge 1}, H_{n\le -1})
    \quad \text{and} \quad
    (\mathbb{C}[t_1,t_3,\ldots], \partial/\partial t_n, (n/2)t_n).
\end{equation*}
Moreover, the fermionic action on the Fock space is intertwined with the action of the vertex operator
\begin{equation}
V(z) = V^{-}(z)V^{+}(z),
\end{equation}
on $\mathbb{C}[t_1,t_3,\dots]$ via the relation
\begin{equation*}
\sigma^i \circ \phi(z) = 2^{-\frac{1}{2}}V(z) \circ \sigma^{1-i.},
\end{equation*}
where the components of $V(z)$ are given by
\begin{align*}
V^{-}(z) = \exp\left(2\sum_{n\ge 1, \text{odd}} \frac{z^n}{n}\frac{n}{2}t_n\right)= \exp\xi(t,z), \quad
V^{+}(z) = \exp\left(2\sum_{n\ge 1, \text{odd}} \frac{z^{-n}}{-n} \frac{\partial}{\partial t_{n}}\right).
\end{align*} 
\end{theorem}

\section{Shifted Schur Measures in Terms of Neutral Fermions}
\label{sec:SS_measure}
A \emph{strict partition} is a strictly decreasing sequence of strictly positive integers $\lambda = (\lambda_1,\ldots,\lambda_l)$, i.e., $\lambda_1 > \dots > \lambda_l > 0$. 
We denote the set of all strict partitions by $\mathcal{D}$. For each $\lambda \in \mathcal{D}$, we define  
\begin{equation*}
    \alpha(\lambda) 
    = \begin{cases}
        0 & l(\lambda) \text{ is even},\\
        1 & l(\lambda) \text{ is odd},
    \end{cases}
\end{equation*}
where $l(\lambda)=l$ signifies the length of the partition.

Recalling that $|1\rangle=2^{\frac{1}{2}}\phi_{0}|0\rangle$, it follows that for any $\lambda\in \mathcal{D}$, the state $\phi_{\lambda_1}\cdots \phi_{\lambda_l} |\alpha(\lambda)\rangle$ lies within the even component $\mathcal{F}_{B}^{0}$ of the Fock space. 
We define the \emph{Schur $Q$-function} $Q_{\lambda}(t)$ by 
\begin{equation}
    Q_{\lambda}(t) 
    = 2^{\frac{l(\lambda)}{2}} \sigma_B^0(\phi_{\lambda_1}\cdots\phi_{\lambda_l} |\alpha(\lambda)\rangle)\rvert_{t=2t} = 2^{\frac{l(\lambda)}{2}} \langle 0 | e^{H_+(2t)}\phi_{\lambda_1}\cdots\phi_{\lambda_l} |\alpha(\lambda)\rangle,
    \label{eq:SchurQ}
\end{equation}
where the second equality is a consequence of \eqref{map_BScor}.
Furthermore, we set $P_{\lambda}(t) = 2^{-l(\lambda)}Q_{\lambda}(t)$.

A straightforward calculation using Wick's theorem yields
\begin{equation*}
    \langle 0| e^{H_{-}(2t)}\phi_{\lambda_1}\cdots\phi_{\lambda_l} |\alpha(\lambda)\rangle
    = (-1)^{|\lambda|}\langle \alpha(\lambda)|\phi_{-\lambda_l}\cdots\phi_{-\lambda_1}e^{H_{-}(2t)}|0\rangle,
\end{equation*}
where $|\lambda|=\lambda_1+\cdots+\lambda_l$ denotes the size of the partition and $H_{-}(t) = \sum_{n\ge 1,\text{odd}} t_n H_{-n}$. 
This relation implies that the value of $Q_\lambda(t)$ is proportional to the coefficient of the basis vector $\phi_{\lambda_1}\cdots\phi_{\lambda_l}|\alpha(\lambda)\rangle$ in the expansion of the state $e^{H_{-}(2t)}|0\rangle$.

The state $e^{H_{-}(2t)}|0\rangle \in \mathcal{F}_B^0$ is a \emph{coherent state}, in the sense that it is a simultaneous eigenvector of the operators $H_n\,(n\ge1)$. 
Indeed, using \eqref{eq:Hamilton}, one finds the eigenvalue relation
\begin{equation*}
    H_n e^{H_{-}(2t)}|0\rangle = nt_n e^{H_{-}(2t)}|0\rangle,
\end{equation*}
whilst the size is given by
\begin{equation*}
    \langle 0|e^{H_+(2\overline{t})}e^{H_-(2t)}|0 \rangle = e^{[H_+(2\overline{t}),H_-(2t)]} = \exp{\textstyle\left(\sum_{n\ge 1, \text{odd}}2n\, |t_n|^2\right)},
\end{equation*}
where $\overline{t}=(\overline{t}_1,\overline{t}_3,\dots)$ denotes the complex conjugate of $t$.
Consequently, the assignment
\begin{equation*}
    \mathcal{D} \ni \lambda \mapsto \frac{|\langle 0| e^{H_+(2t)}\phi_{\lambda_1}\cdots\phi_{\lambda_l} |\alpha(\lambda)\rangle|^2}{\langle 0|e^{H_+(2\overline{t})}e^{H_-(2t)}|0 \rangle}
    = \frac{P_\lambda(\overline{t})Q_\lambda(t)}{\exp{\left(\sum_{n\ge 1, \text{odd}}2n\, |t_n|^2\right)}}
\end{equation*}
defines a probability measure on $\mathcal{D}$, which we call the (Hermitian) \emph{shifted Schur measure}.
We denote this measure by $\mathbb{P}_{\rm SS}(\lambda;t) = \frac{1}{Z}P_{\lambda}(\overline{t})Q_\lambda(t)$, with the partition function $Z=\exp\left(\sum_{n\ge1,\text{odd}}2n\,|t_n|^2\right)$.
 
For a finite set $A \subset \mathbb{Z}_{>0}$, we define the \emph{correlation function}, 
\begin{equation}
    \rho(A;t) = \mathbb{P}_{\rm{SS}}(\{\lambda \in \mathcal{D} \mid A \subset \lambda \};t)
    = \frac{1}{Z} \sum_{\lambda \supset A} P_\lambda(\overline{t}) Q_{\lambda}(t).
    \label{eq:rho}
\end{equation}
In terms of the neutral fermions, $\mathbb{P}_{\rm{SS}}(\lambda;t)$ represents the probability that the fermions occupy exactly the sites corresponding to $\lambda$ in the coherent state $e^{H_+(2t)}|0\rangle$, whereas the correlation function $\rho(A;t)$ signifies the probability that at least the sites in $A$ are occupied. 
Let us define the correlation kernel $K(a,b;t)$ by the vacuum expectation 
\begin{equation*}
    K(a,b;t) = \langle 0 | e^{H_+(2\overline{t})}e^{-H_-(2t)}\phi_a\phi_be^{H_-(2t)}e^{-H_+(2\overline{t})} |0\rangle.
\end{equation*}
For a set $A=\{a_1,\ldots,a_N\}$ of size $N$, we introduce a skew-symmetric matrix $M(A;t)$ of order $2N$, whose entries are prescribed by
\begin{align*}
    M(A;t)_{i,j} =
    \begin{cases}
        K(a_i,a_j;t) & 0<i<j\le N,\\
        (-1)^{a_{2N-j+1}}K(a_i,-a_{2N-j+1};t) & 0<i\le N < j \le 2N,\\
        (-1)^{a_{2N-i+1}+a_{2N-j+1}}K(-a_{2N-i+1},-a_{2N-j+1};t) & N<i<j\le 2N.
    \end{cases}
\end{align*}

\begin{theorem}[\cite{{MR2139724},{WANG2019447}}]
\label{thm:PfPP}
The correlation function \eqref{eq:rho} is given by a Pfaffian,
\begin{equation*}
    \rho(A;t) = \mathrm{Pf} (M(A;t)),
\end{equation*}
hence the shifted Schur measure induces a \emph{Pfaffian point process}.
\end{theorem}

To clarify the structural origin of the matrix $M(A;t)$, we observe that the correlation function admits the representation 
\begin{align*}
    \rho(A) &= \langle 0|\Phi_{a_1}(-1)^{a_1}\Phi_{-a_1}\cdots\Phi_{a_N}(-1)^{a_N}\Phi_{-a_N} |0 \rangle\\
    &= \langle 0|\Phi_{a_1}\cdots\Phi_{a_N}(-1)^{a_N}\Phi_{-a_N}\cdots(-1)^{a_1}\Phi_{-a_1}|0 \rangle,
\end{align*}
where $\Phi_a=G \phi_a G^{-1}$ with $G=e^{H_{+}(2\overline{t})}e^{-H_{-}(2t)}$ (cf. \cite[Theorem 3.3]{WANG2019447}).  
Since each $\Phi_a$ is a linear combination of neutral fermions $\{\phi_{a}\}_{a\in\mathbb{Z}}$, it naturally decomposes into creation and annihilation parts, alongside the zero mode $\phi_{0}$.
By defining the sequence of operators
\begin{equation*}
    (\eta_1,\ldots,\eta_{2N})=(\Phi_{a_1},\ldots,\Phi_{a_N},(-1)^{a_N}\Phi_{-a_N},\ldots,(-1)^{a_1}\Phi_{-a_1})
\end{equation*}
and applying Wick's theorem, we obtain
\begin{equation*}
    \rho(A;t) 
    = \sum_{\substack{\sigma \in \mathfrak{S}_{2N},\\\sigma(1)<\sigma(3)<\cdots<\sigma(2N-1),\\\sigma(2i-1)<\sigma(2i)
    }}
    \mathrm{sgn}(\sigma)\prod_{i=1}^N \langle 0|\eta_{\sigma(2i-1)}\eta_{\sigma(2i)}|0 \rangle.
\end{equation*}
The matrix $M(A;t)$ is then recovered as the skew-symmetric extension of the values $\langle 0|\eta_i\eta_j|0 \rangle$ for $i<j$.

Notably, the correlation kernel $K(a,b;t)$ admits a double contour integral expression that is amenable to asymptotic analysis.
Following the approach in \cite{MR2139724}, we introduce an auxiliary function, which we call the wave function 
\begin{equation}
    \mathbb{J}(z;t) = e^{2\xi(\overline{t},z)}e^{2\xi(t,-z^{-1})}.
\label{eq:mathbbJ}
\end{equation}
Using the transformation property $G\phi(z)G^{-1} = \mathbb{J}(z)\phi(z)$, the kernel is expressed as
\begin{align}
    K(a,b;t) &= \frac{1}{(2\pi \rm{i})^2}\oiint_{|z|>|w|} \langle 0|G\phi(z)\phi(w)G^{-1}|0 \rangle \frac{{\rm d}z {\rm d}w}{z^{a+1} w^{b+1}} \notag\\
    &= \frac{1}{(2\pi \rm{i})^2}\oiint_{|z|>|w|} \mathbb{J}(z)\mathbb{J}(w) \langle 0|\phi(z)\phi(w)|0 \rangle \frac{{\rm d}z{\rm d}w}{z^{a+1} w^{b+1}} \notag\\
    &= \frac{1}{2(2\pi \rm{i})^2}\oiint_{|z|>|w|} \mathbb{J}(z)\mathbb{J}(w) \frac{z-w}{z+w} \frac{{\rm d}z{\rm d}w}{z^{a+1} w^{b+1}},
    \label{eq:K}
\end{align}
where the final equality is a direct consequence of the anti-commutation relations \eqref{eq:R_neu}.

\section{Limit Shape Analysis}\label{sec:lim_shape}
In this section, we introduce the shifted Young diagram, an effective tool to visualise configurations of neutral fermions. 

To construct a shifted Young diagram for a strict partition $\lambda$, we take a standard Young diagram and indent its $i$-th row to the right by $i-1$ boxes. 
To facilitate our limit shape analysis, we transpose this diagram, rotate it by $135^\circ$, and re-orient it, as illustrated in \autoref{fig:YD_strict}.

\begin{figure}[thbp]
\centering
\includegraphics[width=110mm]{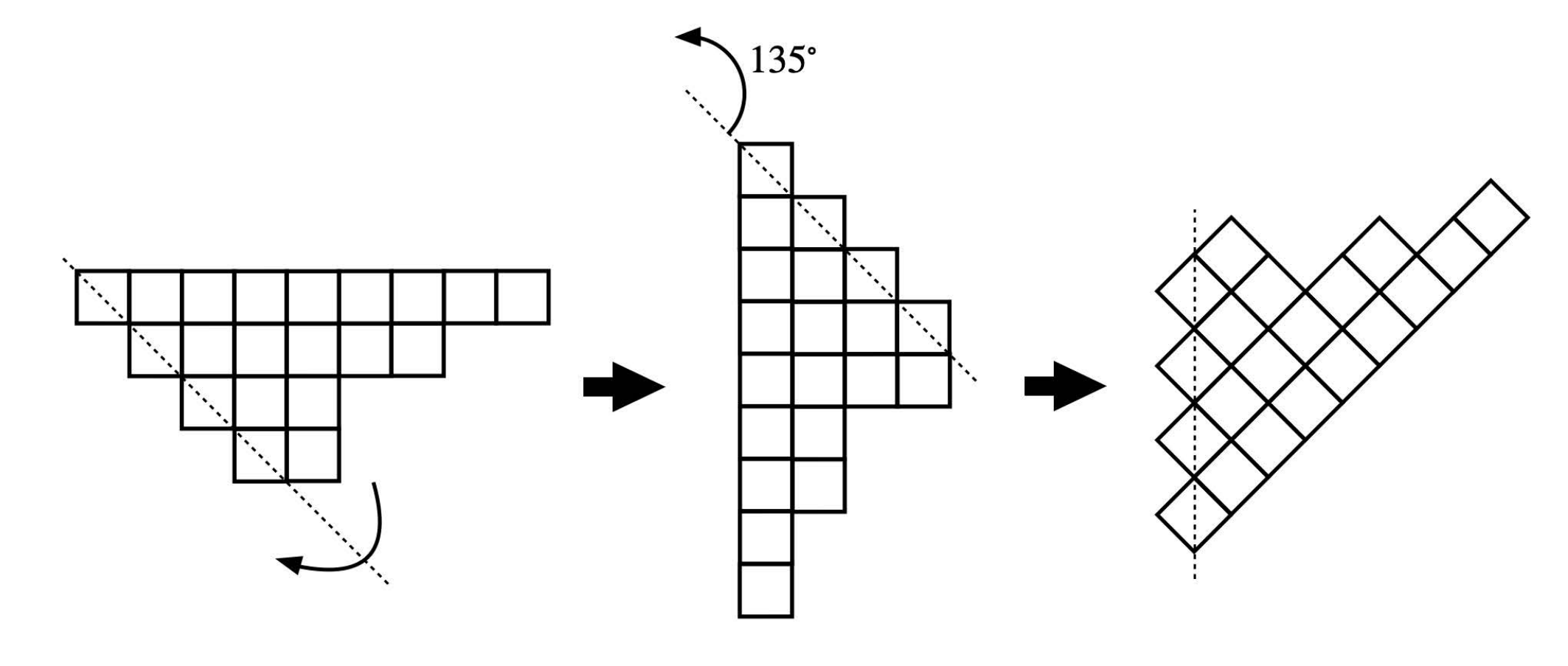}
\caption{Shifted Young diagram for the strict partition $\lambda=(9,6,3,2)$.}
\label{fig:YD_strict}
\end{figure}

We then map the resulting boundary to a sequence of black and white stones on the lattice $\mathbb{Z}_{>0}$ (the \emph{Maya diagram}) following a simple geometric rule: a white stone corresponds to a southeast-oriented boundary segment, whilst a black stone indicates a northeast-oriented one as shown in \autoref{fig:Maya}. 

\begin{figure}[thbp]
\centering
\includegraphics[width=100mm]{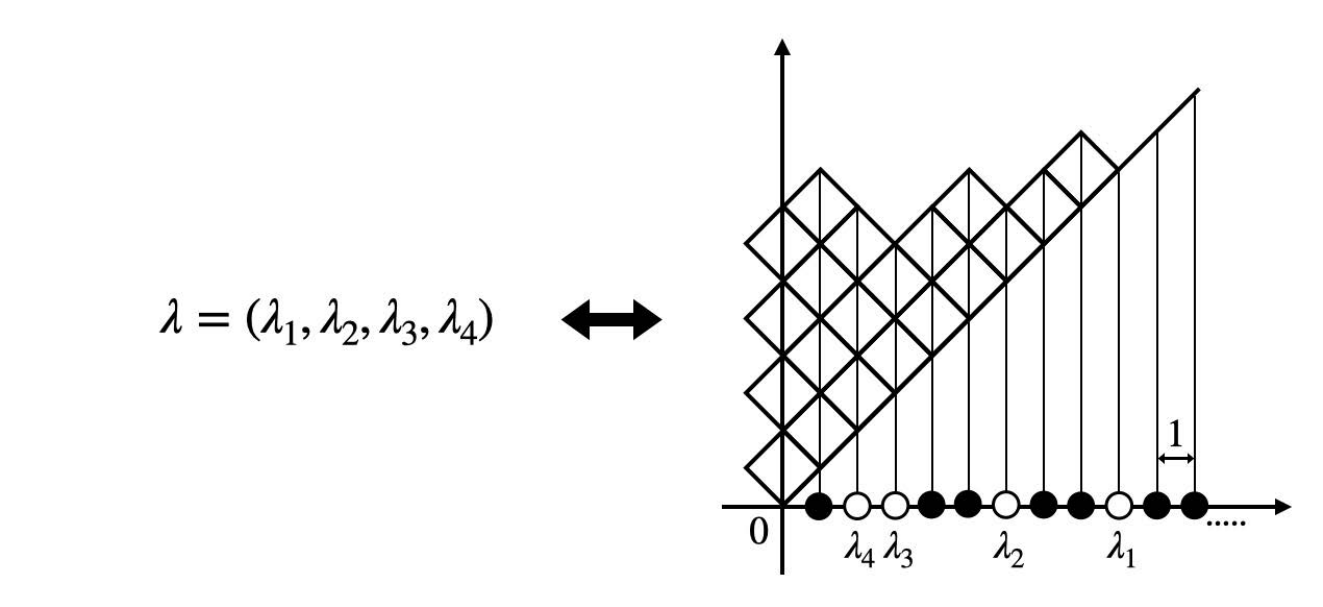}
\caption{The correspondence between a shifted Young diagram and its associated Maya diagram.}
\label{fig:Maya}
\end{figure}

By inspection, for a shifted diagram associated with a strict partition $\lambda=(\lambda_i)$, the positions of the white stones coincide exactly with the parts of the partition $\{\lambda_1,\lambda_2,\dots\}$. 
We associate these white stones with the neutral fermions. 
More formally, the stone configuration corresponds to the even state $\phi_{\lambda_1}\cdots\phi_{\lambda_{l}}|\alpha(\lambda)\rangle \in\mathcal{F}_B^{0}$.

The shifted Young diagram is central to the combinatorial definition of the Schur $Q$-function $Q_\lambda(x_1,x_2,\ldots)$. 
Under the Miwa transformation $t_n = \frac{1}{n} \sum_i x_i^n$, this combinatorial definition neatly coincides with the algebraic expression in \eqref{eq:SchurQ} (see, for instance, \cite{MR2060026}).

We now turn to the limit shape of the shifted Young diagram as the lattice mesh size $\varepsilon$ tends to zero (the scaling limit). 
Concurrently, we scale the Miwa variables $t$ by $1/\varepsilon$, ensuring the number of boxes in the diagram grows at a rate of $O(\varepsilon^{-2})$.

The profile (or height function) of a shifted Young diagram for a strict partition $\lambda$ is characterised by:
\begin{equation*}
    \psi_{\lambda,1}(x)=x+2\#\{i\mid x<\lambda_i\}, \quad x\in \mathbb{Z}_{\ge 0}.
\end{equation*}
Rescaling the box size by $\varepsilon$ yields the profile function:
\begin{equation*}
    \psi_{\lambda,\varepsilon}(x)=x+2\varepsilon\#\{i\mid x<\lambda_i\varepsilon\}, \quad x\in \varepsilon\mathbb{Z}_{\ge 0}.
\end{equation*}
Our primary interest lies in its expectation value, which can be expressed in terms of the one-point correlation function:
\begin{equation*}
    \mathbb{E}[\psi_{\lambda,\varepsilon}(x)] 
    = x + 2\varepsilon \sum_{k=1}^{\infty} \rho\left(\frac{x}{\varepsilon}+k;\frac{t}{\varepsilon}\right).
\end{equation*}
This expectation is taken with respect to the scaled shifted Schur measure $\mathbb{P}_{\rm SS}^{\varepsilon}(\lambda) = \mathbb{P}_{\rm SS}(\lambda;t/\varepsilon)$.

\subsubsection*{Constraints on the Miwa Variables}

To ensure a well-defined limit shape with specific critical behaviour, we assume that the Miwa variables $t$ satisfy the following conditions:

\begin{enumerate}[label=(\roman*)]
    \item \textbf{Real-valuedness:} $t_1, t_3, \dots$ are all real numbers.
    
    \item \textbf{Finite support:} $t_n = 0$ except for a finite number of $n$.
    
    \item \textbf{$p$-Multicriticality:} There exists a positive constant $a > 0$ such that the function 
    \begin{equation*}
        \varphi(\theta) = 4\sum_{n>0, \text{odd}} t_n \sin n\theta - a\theta, 
    \end{equation*}
    possesses a Taylor expansion at $\theta = 0$ of the form 
    \begin{equation}
        \varphi(\theta) = -\frac{\theta^{p+1}}{p+1} + O(\theta^{p+3}), 
        \label{eq:multicritical}
    \end{equation}
    for a given positive even integer $p$.
    
    \item \textbf{Non-degeneracy:} The variables satisfy
    \begin{equation*}
        4\sum_{n>0, \text{odd}} n^2 t_n \sin n\theta > 0 \quad \text{for all } \theta \in (0, \pi).
    \end{equation*}
\end{enumerate}
%\item[(iv)' Non-degeneracy for Soft Edge] $\varphi'(\theta)=0 \Leftrightarrow \theta=2n\pi \;(n\in\mathbb{Z})$
%We call these conditions on the parameters the ($p$-)multicritical condition. 

These constraints mirror the conditions established in \cite{MR4693927} for the multicritical Schur measure. 
In particular, condition (iii) represents a system of $p/2$ linear equations and a linear inequality in $t$. 
By choosing $t_1,t_3,\ldots,t_{p-1}$ to satisfy these equations and setting $t_{p+1}=t_{p+3}=\cdots=0$, the variables naturally satisfy conditions (i)–(iv). 
We refer to this specific configuration as the \emph{minimal $p$-multicritical shifted Schur measure}. 
For instance, when $p=2$, the choice $t_1=1/2$, $t_{n>1}=0$, and $a=2$ neatly satisfies all the above criteria.

Before proceeding to the limit shape analysis, we outline several consequences of conditions (i)--(iv) that will be essential for our later derivations.

\subsection*{Consequences of the Constraints}

First, assume the Miwa variables $t$ satisfy conditions (i)–(iii). 
The Taylor expansion \eqref{eq:multicritical} implies that $\varphi^{(k)}(0)=0$ for $k=0,1,\dots p$. 
Specifically, the condition $\varphi'(0)=0$ leads to the identity:
\begin{equation*}
    a = 4\sum_{n:\text{odd}} n t_n,
\end{equation*}
which subsequently yields $\varphi'(\pi)=-2a$. 
If the non-degeneracy condition (iv) also holds, $\varphi'(\theta)$ decreases monotonically from $\varphi'(0)=0$ to $\varphi'(\pi)=-2a$ (assuming a multicritical setting with a fixed $a$). 
Moreover, there exists a constant $C>1$ such that
\begin{equation}
    \varphi'(\theta) \le -\frac{\theta^p}{C}\quad\text{on} -\pi \le \theta \le \pi.
    \label{eq:C}
\end{equation}

Next, consider the case where $t$ satisfies (i), (ii), and (iv). 
We define the function $D(\theta)=4\sum_{n:\text{odd}}nt_n\cos{n\theta}$. 
Since $D(\theta)$ decreases monotonically from $D(0)=b$ to $D(\pi)=-b$, where
\begin{equation*}
    b\coloneqq 4\sum_{n:\text{odd}}nt_n > 0,
\end{equation*}
it follows that for any $x\in[-b,b]$, there exists a unique $\chi(x)\in [0,\pi]$ such that
\begin{equation}
D(\chi(x))=x.
\label{eq:chi}
\end{equation}
From this point onward, we take conditions (i) and (ii) as our baseline assumptions for the Miwa variables $t$.

\subsection*{Integral Representation of the Profile Function}

With these preparations, we return to the expectation value of the profile function.
The Pfaffian expression for the correlation function (Theorem \ref{thm:PfPP}) allows us to write the scaled one-point function as follows,
\begin{align} \label{eq:rho_ontheway}
    \rho\left(\frac{x}{\varepsilon}+k;\frac{t}{\varepsilon}\right)
    &= \frac{(-1)^{x/\varepsilon+k}}{2(2\pi\rm{i})^2}\oiint_{|z|>|w|}\mathbb{J}(z)\mathbb{J}(w)\frac{z-w}{z+w}\frac{{\rm d}z{\rm d}w}{z^{x/\varepsilon+k+1}w^{-x/\varepsilon-k+1}} \\
    &= \frac{(-1)^{x/\varepsilon+k}}{2(2\pi\rm{i})^2}\oiint_{|z|>|w|} e^{\frac{1}{\varepsilon}[S(z,x)+S(w,-x)]}\frac{z-w}{z+w}\frac{{\rm d}z{\rm d}w}{z^{k+1}w^{-k+1}}, \notag  
\end{align}
where the action function $S(z,x)$ is defined by
\begin{equation}
    S(z,x)=\xi(2t,z)+\xi(2t,-z^{-1})-x\log z.
\label{eq:S(z,x)}
\end{equation}
Utilising the identity $e^{S(w,-x)/\varepsilon}=e^{-S(-w,x)/\varepsilon} (-1)^{x/\varepsilon}\,\,(x/\varepsilon \in \mathbb{Z})$, the expression \eqref{eq:rho_ontheway} simplifies to
\begin{align}
    \rho\left(\frac{x}{\varepsilon}+k; \frac{t}{\varepsilon}\right)
    &= \frac{(-1)^{k}}{2(2\pi\rm{i})^2}\oiint_{|z|>|w|} e^{\frac{1}{\varepsilon}[S(z,x)-S(-w,x)]}\frac{z-w}{z+w}\frac{{\rm d}z{\rm d}w}{z^{k+1}w^{-k+1}}  \notag\\
    &= \frac{1}{2(2\pi\rm{i})^2}\oiint_{|z|>|w|} e^{\frac{1}{\varepsilon}[S(z,x)-S(w,x)]}\frac{z+w}{z-w}\frac{{\rm d}z{\rm d}w}{z^{k+1}w^{-k+1}}. 
    \label{eq:chamu}
\end{align}
Finally, by summing over $k \in \{ 1,2,\dots \}$, we arrive at a double contour integral representation for the expectation of the profile function:
\begin{equation}
2 \varepsilon \sum_{k=1}^{\infty} \rho\left(\frac{x}{\varepsilon}+k ;\frac{t}{\varepsilon}\right)
=\frac{\varepsilon}{(2\pi\rm{i})^2}\oiint_{|z|>|w|}e^{\frac{1}{\varepsilon}[S(z,x)-S(w,x)]}\frac{z+w}{z(z-w)^2}{\rm d}z{\rm d}w.
\label{eq:profile}
\end{equation}

The following theorem characterises the deterministic limit of the profile function under the bulk scaling limit.

\begin{theorem}
\label{thm:limit_shape}
    Let $b = 4\sum_{n:\text{odd}}nt_n$.
    Suppose the Miwa variables $t$ satisfy the conditions (i),(ii) and (iv). Then, we have 
    \begin{equation}
        \lim_{\varepsilon \to 0} 2\varepsilon\sum_{k=1}^{\infty} \rho\left(\frac{x}{\varepsilon}+k; \frac{t}{\varepsilon}\right) 
        =
        \begin{cases}
            \frac{2}{\pi}\int_0^{\chi(x)}(D(\theta)-x){\rm d}\theta & 0\le x\le b,\\
            0 & b \le x. 
        \end{cases}
        \label{eq:rho_limit}
    \end{equation}
    It follows that the limit shape of the shifted Young diagram is described by the deterministic function
    \begin{equation*}
        \lim_{\varepsilon\to 0} \mathbb{E}[\psi_{\lambda,\varepsilon}(x)] = 
        \begin{cases}
        x+\frac{2}{\pi}\int_0^{\chi(x)}(D(\theta)-x){\rm d}\theta & 0\le x\le b\\
        x & b\le x.
        \end{cases}
    \end{equation*}
\end{theorem}

To gain further intuition, let us consider specific choices for the Miwa variables.

\begin{example}
Consider the standard case where $t_1=1/2, t_3=t_5=\cdots=0$ and $a=b=2$.
In this setting, the relation \eqref{eq:chi} yields $\chi(x)=\arccos{x/2}$. 
A straightforward integration then yields the limit shape:
\begin{equation*}
    \lim_{\varepsilon\to 0} \mathbb{E}[\psi_{\lambda,\varepsilon}(x)] = 
    \begin{cases}
        x+\frac{2}{\pi}(\sqrt{4-x^2}-x\arccos{\frac{x}{2}}) & 0\le x\le 2\\
        x & 2\le x
    \end{cases}
\end{equation*}
which is the Logan-Shepp-Vershik-Kerov curve.

More generally, we can investigate higher multicritical regimes. The limit shapes and the corresponding fermion densities for the minimal $p$-multicritical measures ($p=2,4,\dots,20$) are visualised in \autoref{fig:Density_Shape}.

\begin{figure}[tbp]
    \centering

    % 左の図
    \begin{subfigure}[t]{0.48\textwidth} % 少し幅を広げても大丈夫
        \centering
        \begin{tikzpicture}
            % 画像
            \node[anchor=south west, inner sep=0] (img1) at (0,0) {%
                \includegraphics[height=18em]{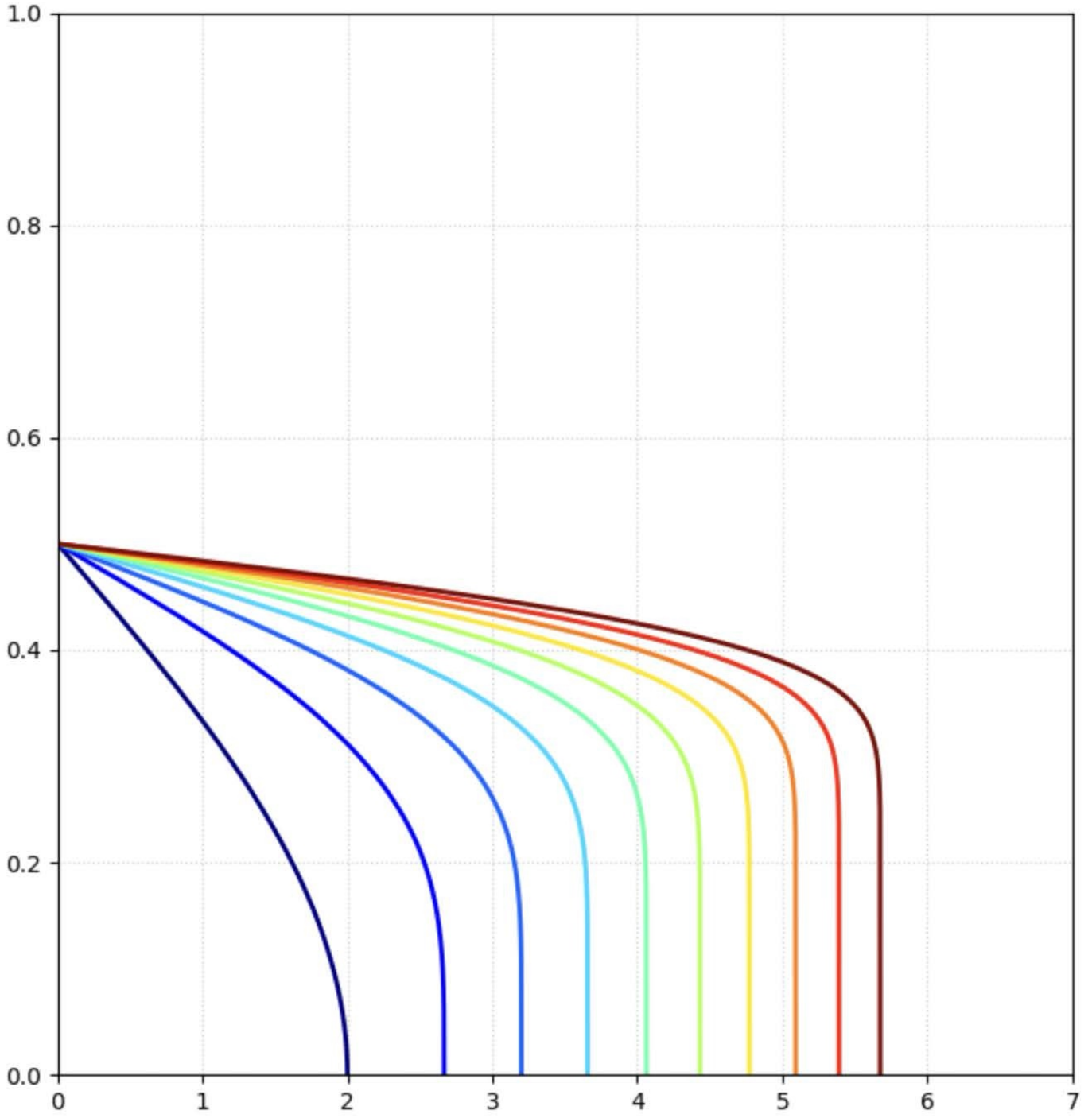}%
            };

            % 相対座標系
            \begin{scope}[x={(img1.south east)}, y={(img1.north west)}]
                % x軸ラベル 
                \node[below] at (0.45, 0) {$x$};
                % y軸ラベル 
                \node[left, rotate=90, anchor=south] at (0, 0.5) {$\chi/\pi$};
            \end{scope}
        \end{tikzpicture}
        \caption{Density function in the scaling limit. }
    \end{subfigure}
    \hfill % 図の間の空行を消す
    % 右の図
    \begin{subfigure}[t]{0.48\textwidth}
        \centering
        \begin{tikzpicture}
            % 画像
            \node[anchor=south west, inner sep=0] (img2) at (0,0) {%
                \includegraphics[height=18em]{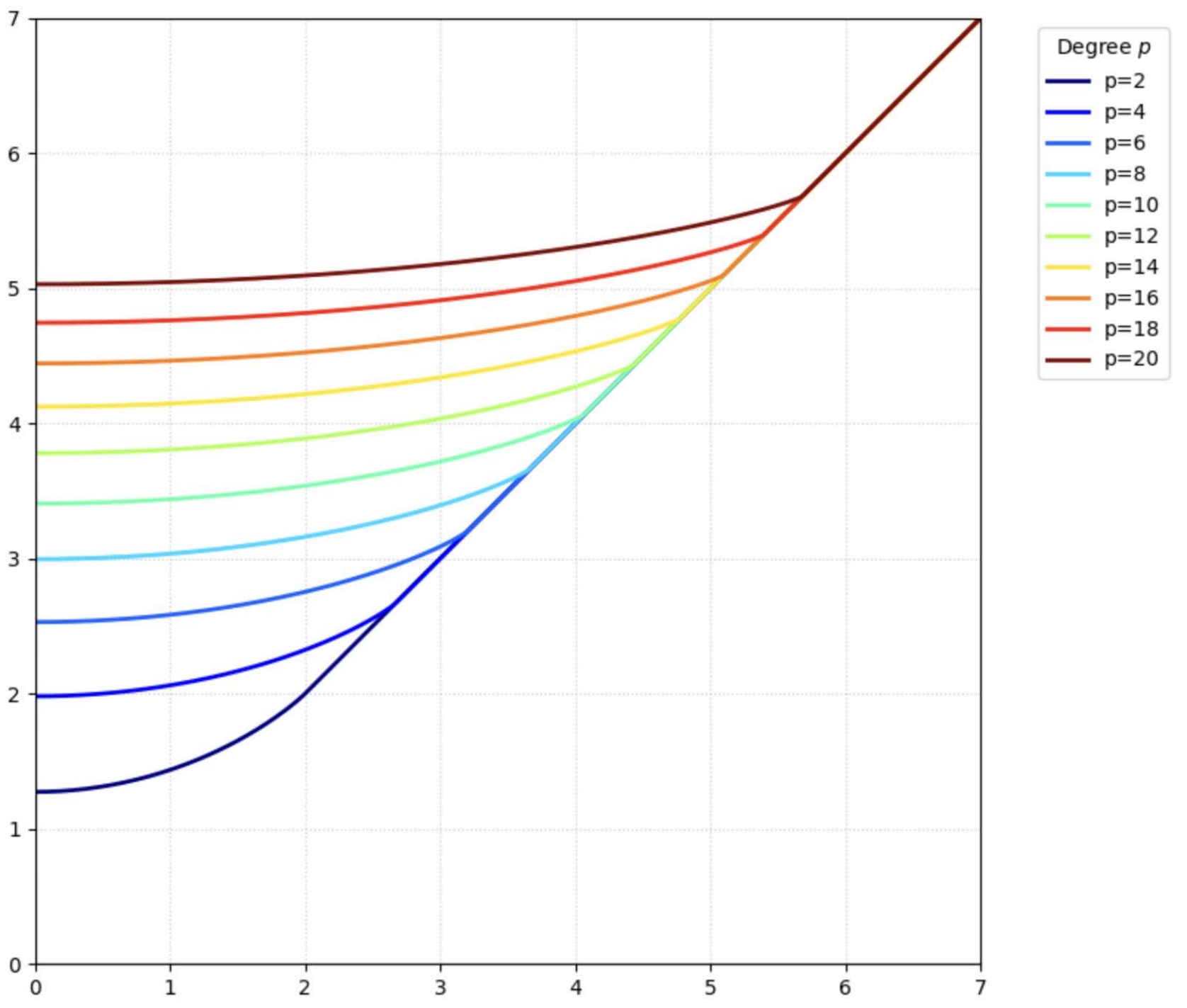}%
            };

            % 相対座標系
            \begin{scope}[x={(img2.south east)}, y={(img2.north west)}]
                % x軸ラベル
                \node[below] at (0.45, 0) {$x$};
                % y軸ラベル
                \node[left, rotate=90, anchor=south] at (0, 0.5) {$\displaystyle \lim_{\varepsilon\to 0} \mathbb{E}[\psi_{\lambda,\varepsilon}(x)]$};
            \end{scope}
        \end{tikzpicture}
        \caption{The limit shape for minimal $p$-multicritical measures.}
    \end{subfigure}

    \vspace{1em} % 全体キャプションとの間に少し隙間を作る
    \caption{Density and Limit Shape}
    \label{fig:Density_Shape}
\end{figure}
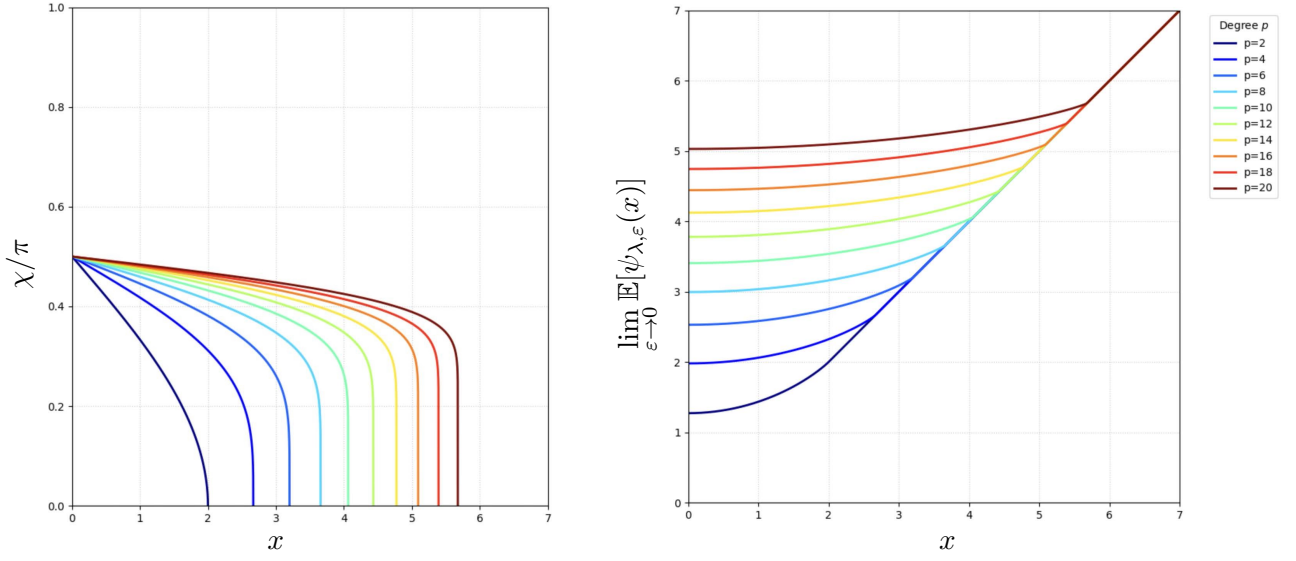

\end{example}

\begin{proof}[Proof of \autoref{thm:limit_shape}] 
We restrict our attention to the case $0\le x\le b$, as the logic for $x\ge b$ follows naturally from identical arguments.

We deform the contour of integration in \eqref{eq:profile} so that the integral vanishes as $\varepsilon \to 0$. 
This deformation inevitably picks up a residue from the $\frac{1}{(z-w)^2}$ pole in the integrand, which precisely yields the right-hand side of \eqref{eq:rho_limit}.

Differentiating $S(re^{{\rm i}\theta},x)$ with respect to $r$ at $r=1$ yields
\begin{equation*}
        \operatorname{Re}S(re^{{\rm i}\theta},x) = (r-1)[D(\theta)-x] + O((r-1)^2). 
\end{equation*}
Thus, by deforming the integration contours to $c_z$ and $c_w$ as are illustrated in \autoref{fig:Path_c_zw},
one can ensure that $\operatorname{Re}[S(z,x)-S(w,x)] < 0$ everywhere except at the points $z,w=e^{\pm {\rm i}\chi(x)}$.  
    
We carry out this deformation as follows. 
Starting from the integration paths $\abs{z}\equiv\text{const.}>1$ and $\abs{w}\equiv1$, we rewrite the integral \eqref{eq:profile} as a repeated integral $\int{\rm d}w\int{\rm d}z \cdots$. 
First, we deform the path for $z$ to $c_z$.
We partition the path for $w$ into $c_{\chi}$ and $c'_{\chi}$, as illustrated in \autoref{fig:Path_c_chi}. 
\begin{figure}[thbp]
    \centering
    \includegraphics[width=59mm]{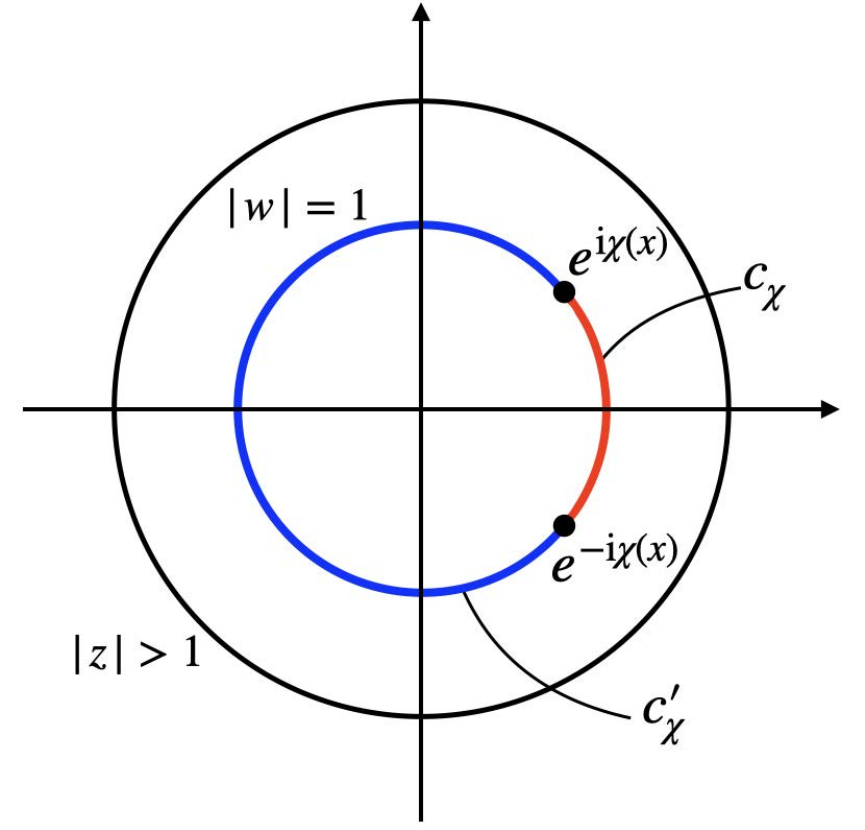}
    \caption{$c_{\chi}$ and $c'_{\chi}$}
    \label{fig:Path_c_chi}
    \end{figure}
When $w \in c_{\chi}$, the contour for $z$ is forced to cross the point $w$, which incurs a residue (see \autoref{fig:Path_cz}):
\begin{equation*}
    \operatorname*{Res}_{z=w} \left\{e^{\frac{1}{\varepsilon} S(z,x)-S(w,x)} \frac{z+w}{z(z-w)^2} \right\}
    = \frac{2}{\varepsilon} S'(w,x) -\frac{1}{w}.
    \end{equation*}
Conversely, when $w \in c'_{\chi}$, no such obstruction occurs.   
    
    \begin{figure}[thbp]
    \centering
    \begin{subfigure}[b]{0.40\columnwidth}
    \centering
    \includegraphics[width=0.9\columnwidth]{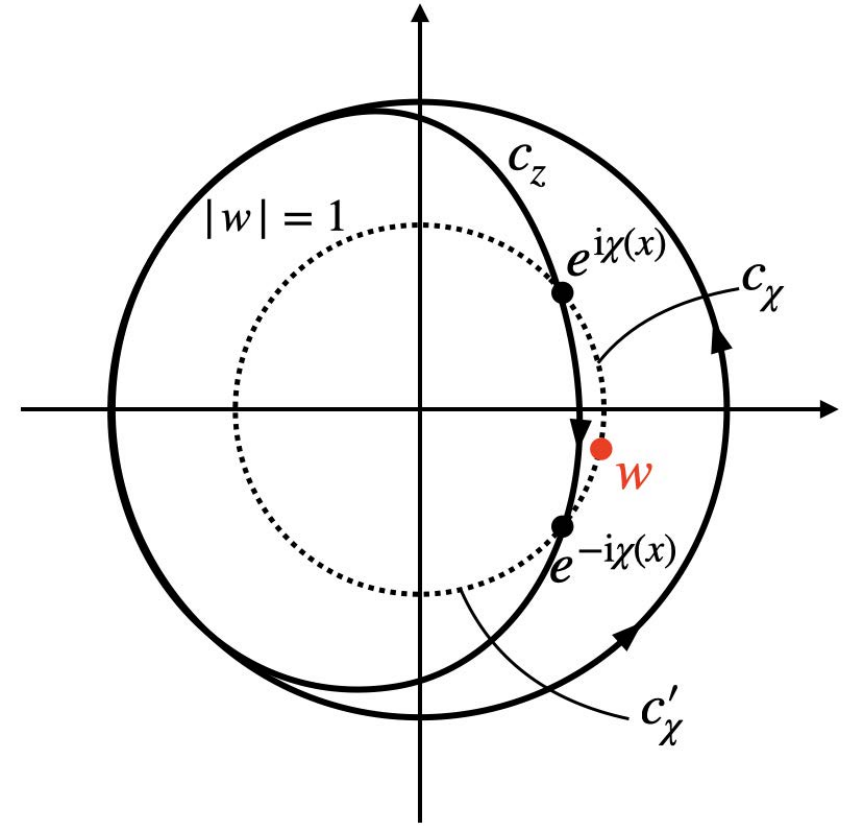}
    \caption{$w \in c_{\chi}$}
    \label{fig:Path_cz_c_chi}
    \end{subfigure}
    \begin{subfigure}[b]{0.40\columnwidth}
    \centering
    \includegraphics[width=0.9\columnwidth]{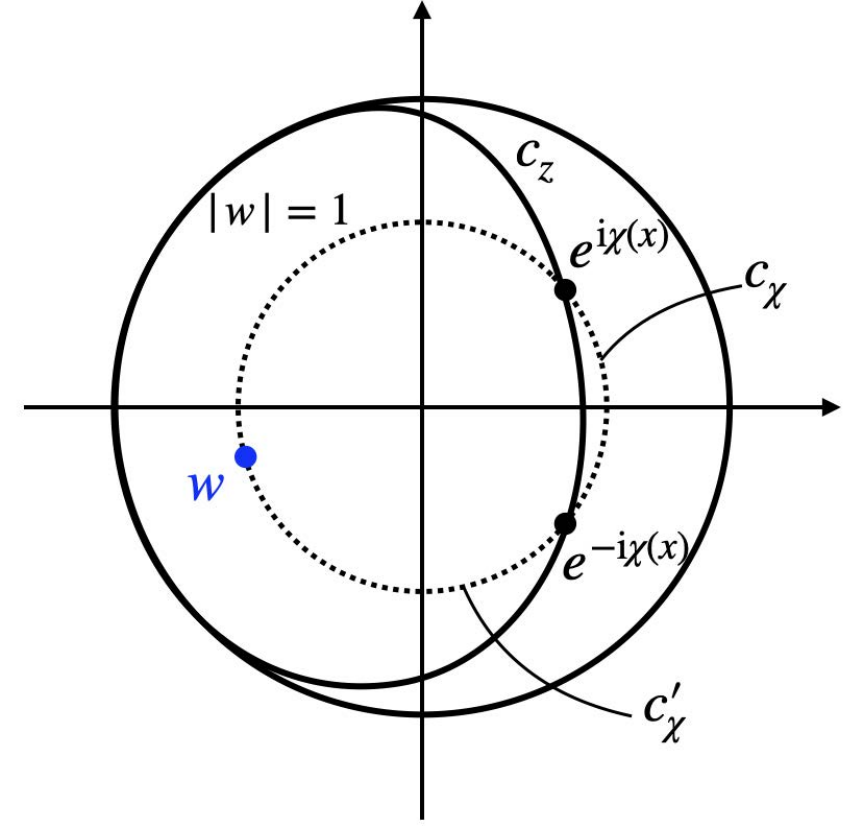}
    \caption{$w \in c'_{\chi}$}
    \label{fig:Path_cz_c'_chi}
    \end{subfigure}
    \caption{Deformation of $c_z$.}
    \label{fig:Path_cz}
    \end{figure} 

With the deformation for $z$ complete, we subsequently deform the path for $w$ to $c_w$ as shown in \autoref{fig:Path_c_zw}. 
This step proceeds without obstruction. 
    
    \begin{figure}[thbp]
    \centering
    \includegraphics[width=59mm]{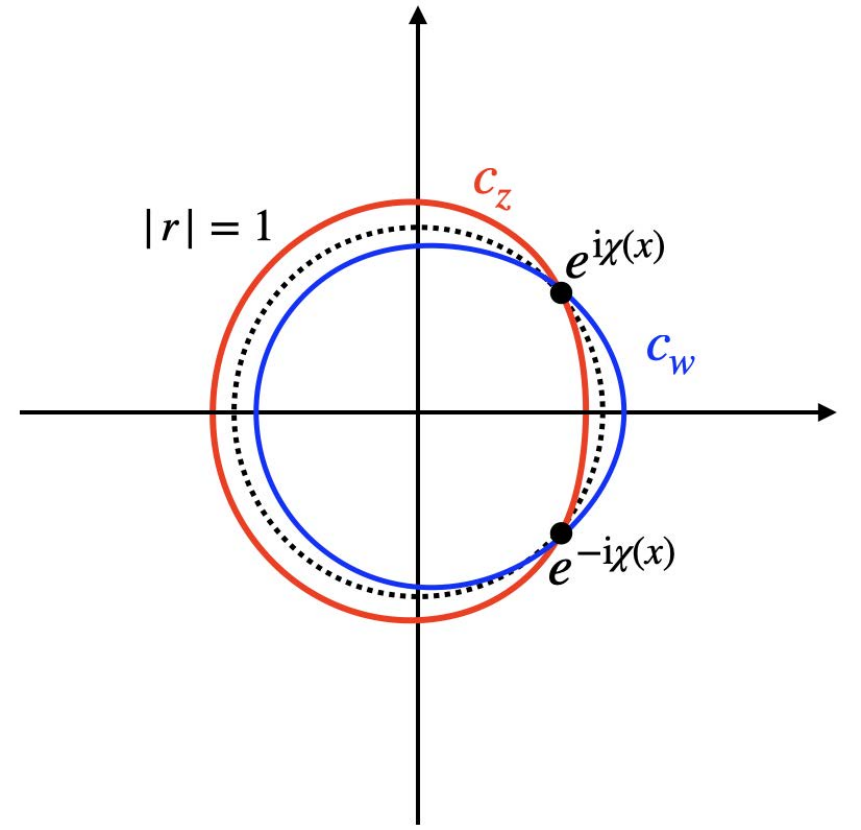}
    \caption{$c_{z}$ and $c_{w}$}
    \label{fig:Path_c_zw}
    \end{figure}

Summarizing the deformation procedure, we obtain 
\begin{align*}
     2 \sum_{k=1}^{\infty} \rho\left(\frac{x}{\varepsilon}+k ;\frac{t}{\varepsilon}\right)
        &=\frac{1}{(2\pi\rm{i})^2}\oiint_{|z|>|w|}e^{\frac{1}{\varepsilon}[S(z,x)-S(w,x)]}\frac{z+w}{z(z-w)^2}{\rm d}z{\rm d}w  \notag\\
        &= \frac{1}{2\pi\rm{i}}\int_{w\in c_\chi}
        \operatorname*{Res}_{z=w}
        e^{\frac{1}{\varepsilon}[S(z,x)-S(w,x)]}\frac{z+w}{z(z-w)^2}{\rm d}w \notag\\
        &\qquad \qquad + \frac{1}{(2\pi\rm{i})^2}\oiint_{z\in c_z,w\in c_w}e^{\frac{1}{\varepsilon}[S(z,x)-S(w,x)]}\frac{z+w}{z(z-w)^2}{\rm d}z{\rm d}w  \notag\\
        &= \frac{1}{2\pi \rm{i}}\int_{c_\chi}\left(\frac{2}{\varepsilon}S'(w,x) - \frac{1}{w}\right){\rm d}w \notag\\
        &\qquad \qquad + \frac{1}{(2\pi\rm{i})^2}\oiint_{z\in c_z,w\in c_w}e^{\frac{1}{\varepsilon}[S(z,x)-S(w,x)]}\frac{z+w}{z(z-w)^2}{\rm d}z{\rm d}w. 
    \end{align*}
   
It follows readily that 
\begin{align*}
    \lim_{\varepsilon\to 0}\frac{\varepsilon}{2\pi \rm{i}}\int_{c_\chi}\left(\frac{2}{\varepsilon}S'(w,x) - \frac{1}{w}\right){\rm d}w 
    &= \frac{2}{\pi}\int_{0}^{\chi(x)}(D(\theta)-x){\rm d}\theta. 
\end{align*}
So our goal is to show 
\begin{equation*}
    \lim_{\varepsilon\to0} \varepsilon\oiint_{z \in c_{z},w \in c_{w}} e^{\frac{1}{\varepsilon}[S(z,x)-S(w,x)]} \frac{z+w}{z(z-w)^2} {\rm{d}}z {\rm{d}}w = 0.
\end{equation*}
Since the integral is analytic in $1/\varepsilon$, the left-hand side, if it exists, is equal to 
\begin{equation*}
    \lim_{\varepsilon \to 0} \frac{\rm d}{{\rm d}(1/\varepsilon)} \oiint_{z \in c_{z},w \in c_{w}} e^{\frac{1}{\varepsilon}[S(z,x)-S(w,x)]} \frac{z+w}{z(z-w)^2} {\rm{d}}z {\rm{d}}w.
\end{equation*}
Differentiating under the integral sign, it suffices to show 
\begin{equation}
    \lim_{\varepsilon \to 0} \oiint_{z \in c_{z},w \in c_{w}} [S(z,x)-S(w,x)] e^{\frac{1}{\varepsilon}[S(z,x)-S(w,x)]} \frac{z+w}{z(z-w)^2} {\rm{d}}z {\rm{d}}w =0.
\label{eq:sufficient}
\end{equation}

Because $S(z,x)$ and $S(w,x)$ are multi-valued, we divide the paths $c_z$ and $c_w$ into sub-arcs $c_z^i$ and $c_w^j$ to establish single-valued branches. 
We select these branches such that their arguments coincide at the intersection points $z=w=e^{\pm{\rm i}\chi(x)}$. 
Consequently, the kernel $[S(z,x)-S(w,x)]\frac{1}{(z-w)^2}$ exhibits a simple pole of order $1$ at $z=w$, rendering it integrable within the iterated integral.
    %Let the latter term be denoted by 
    %\begin{equation}
    %    F(\zeta)=\frac{1}{(2 \pi \rm{i})^2} \oiint_{z \in c_{z},w \in c_{w}} e^{\zeta[S(z,x)-S(w,x)]} \frac{z+w}{z(z-w)^2} {\rm{d}}z {\rm{d}}w.
    %\end{equation}
    %By the definiton of $S(z,x)$ in (\ref{eq:S(z,x)}), this integral is single-valued if $\zeta=1/\varepsilon$ and $x/\varepsilon \in \mathbb{Z}$; otherwise, it is multi-valued. We therefore decompose the paths $c_{z}$ and $c_{w}$ into sub-arcs along which the argument remains constant, and perform the integration with the argument fixed. 
    
   % \begin{proposition} 
    %    $\lim_{\zeta \to \infty} F'(\zeta)=0.$ Especially,
     %   $\lim_{\varepsilon\to+0} \varepsilon F(\frac{1}{\varepsilon})=0.$
      %  \label{prop:F_zeta}
    %\end{proposition}
    %\begin{proof}
    %    Without loss of generality, it sufficies to consider $\zeta=1/\varepsilon \in\mathbb{R}$. 
    %    Let $\{c_{z}^{i}\}_i$ and $\{c_{w}^{j}\}_j$ be sub-arcs of $c_{z}$ and $c_{w}$, respectively.

\case{1} $c_{z}^{i} \cap c_{w}^{j}=\emptyset$.
Let 
\begin{equation*}
    M\coloneqq \sup_{z\in c_{z}^{i},w \in c_{w}^j} \abs{\frac{z+w}{z(z-w)^2}} <\infty.
\end{equation*}
Since $\operatorname{Re} [S(z,x)-S(w,x)]<0$, it follows that
\begin{equation*}
    \abs{
    [S(z,x)-S(w,x)] e^{\frac{1}{\varepsilon}[S(z,x)-S(w,x)]} \frac{z+w}{z(z-w)^2}
    }  
    \le
    \norm{S(z,x)-S(w,x)}_{\infty} \cdot M.
\end{equation*}
Moreover, as $\varepsilon \to +0$,
\begin{equation*}
    \abs{e^{\frac{1}{\varepsilon} [S(z,x)-S(w,x)]}} 
    \le
    e^{\frac{1}{\varepsilon} \operatorname{Re}[S(z,x)-S(w,x)]}
    \to 0. 
\end{equation*}
Therefore, by the bounded convergence theorem, the integral \eqref{eq:sufficient} over these sub-arcs vanishes.

\case{2} $c_{z}^{i} \cap c_{w}^{j} \ne \emptyset$.
Without loss of generality, let us assume $c_{z}^{i} \cap c_{w}^{j} = \{e^{{\rm i}\chi(x)}\}$. 
We parametrise each sub-arc as follows: 
\begin{equation*}
    \begin{cases}
        c_{z}^{i}: z=z(t), & z(0)=e^{{\rm{i}}\chi(x)} \quad (t_{0} \le t \le t_{1}),   \\
        c_{w}^{j}: w=w(s), & w(0)=e^{{\rm{i}}\chi(x)} \quad (s_{0} \le s \le s_{1}).
    \end{cases}
\end{equation*}
Since $S(z(t),x)-S(w(s),x)=O(z(t)-w(s))$, there exists $M>0$ such that
\begin{equation*}
    \abs{[S(z(t),x)-S(w(s),x)]\frac{z(t)+w(s)}{z(t)(z(t)-w(s))^2}} \le \frac{M}{|t|+|s|}.
\end{equation*}
This implies that the integrand is bounded by an $L^1(c_{z}^{i} \times c_{w}^{j})$ function. 
Meanwhile, we have
\begin{equation*}
    \abs{e^{\frac{1}{\varepsilon} [S(z,x)-S(w,x)]}} 
    \le
    e^{\frac{1}{\varepsilon} \operatorname{Re}[S(z,x)-S(w,x)]}
    \to 0 \quad \text{a.e.}
\end{equation*}
Invoking Lebesgue's dominated convergence theorem, we obtain
\begin{equation*}
    \lim_{\varepsilon\to0}
    \frac{1}{(2\pi {\rm{i}})^2} \oiint_{z \in c_{z}^{i}, w \in c_{w}^{j}} [S(z,x)-S(w,x)] e^{\frac{1}{\varepsilon}[S(z,x)-S(w,x)]} \frac{z+w}{z(z-w)^2} {\rm{d}}z {\rm{d}}w =0.
\end{equation*}
This completes the proof.
\end{proof}

\subsection*{Bulk scaling limit}

By arguing precisely the same way as in the proof of Theorem \ref{thm:limit_shape}, 
we obtain the following scaling limit: \[
\lim_{\varepsilon \to 0} (-1)^{x/\varepsilon+t}K\left(\frac{x}{\varepsilon}+r,-\frac{x}{\varepsilon}-s;\frac{t}{\varepsilon}\right) =
\begin{cases}
    \delta_{r,s} & x \le -b\\
    \frac{\sin{\chi(x)(r-s)}}{\pi(r-s)} & -b\le x\le b\\
    0 & b\le x.
\end{cases}
\]

Indeed, the left-hand side is written as \[
(-1)^{x/\varepsilon+t}K\left(\frac{x}{\varepsilon}+r,-\frac{x}{\varepsilon}-s;\frac{t}{\varepsilon}\right) 
= \frac{1}{2(2\pi{\rm i})^2} \oiint_{\abs{z}>\abs{w}} \exp\left(\frac{1}{\varepsilon}[S(z,x)-S(w,x)]\right)\frac{z+w}{z-w}\frac{{\rm d}z{\rm d}w}{z^{r+1}w^{-s+1}}.
\]
When $-b\le x\le b$, for example, deforming the integration contour $\abs{z}>\abs{w}$ to $z\in c_z,w\in c_w$ costs  \begin{align*}
\frac{1}{2\cdot 2\pi {\rm i}}\int_{c_\chi} \operatorname*{Res}_{z=w} \exp\left(\frac{1}{\varepsilon}[S(z,x)-S(w,x)]\right)\frac{z+w}{z^{r+1}w^{-s+1}}{\rm d} w 
&= \frac{1}{2\pi} \int_{-\pi}^{\pi} e^{-{\rm i}(r-s)\theta}{\rm d}\theta\\
&= \frac{\sin{\chi(x)(r-s)}}{\pi(r-s)}.
\end{align*}
On the other hand, the integral on the deformed contour $z\in c_z,w\in c_w$ converges to $0$ essentially because, although the integrand contains a simple pole of order 1 at $z=w$, it is integrable since the integration is repeated twice.
This proves our assertion.

As a corollary, we have the scaling limit of the one-point function,
\begin{equation}
\lim_{\varepsilon \to 0}\rho \left(\frac{x}{\varepsilon};\frac{t}{\varepsilon}\right) = \frac{\chi(x)}{\pi} , \quad 0 \le x \le b .
\label{eq:diagonalbulk}
\end{equation}

Moreover, by deforming both the integration contour for $z$ and $w$ into $c_z$ in (i), or $c_w$ in (ii), we obtain \begin{enumerate}
    \item $\lim_{\varepsilon \to 0} K\left(\frac{x}{\varepsilon}+r,\frac{x}{\varepsilon}+s;\frac{t}{\varepsilon}\right)=0,$
    \item $\lim_{\varepsilon \to 0} (-1)^{\frac{x}{\varepsilon}+r+\frac{x}{\varepsilon}+s}K\left(-\frac{x}{\varepsilon}-r,-\frac{x}{\varepsilon}-s;\frac{t}{\varepsilon}\right)=0.$
\end{enumerate}
Therefore, using Theorem \ref{thm:PfPP}, we obtain the following result.

\begin{proposition}
    For $0 \le x \le b$, we define the sine kernel by $K_{\sin}(r,s) = \frac{\sin{\chi(x)(r-s)}}{\pi(r-s)}$. 
    Then, the correlation function in the bulk scaling limit is given by a determinant of the sine kernel,
    \[
    \lim_{\varepsilon \to 0} \rho\left(\left\{\frac{x}{\varepsilon}+k_i\right\}_{i=1}^N;\frac{t}{\varepsilon}\right) = \det_{1\le i,j\le N} K_{\sin}(k_i,k_j).
    \]
\end{proposition}

\section{Multicritical Edge Scaling Limit}
\label{sec:edge_lim}
In this section, we turn our attention to the edge behaviour of the previously established limit shape. 
By construction, the limit
\begin{equation*}
    \lim_{\varepsilon \to 0} \varepsilon\sum_{k=1}^{\infty} \rho\left(\frac{x}{\varepsilon}+k;\frac{t}{\varepsilon}\right) =
        \frac{1}{\pi}\int_0^{\chi(x)}(D(\theta)-x){\rm d}\theta, \quad 0\le x\le b
\end{equation*}
can be interpreted as the (scaled) expected total number of neutral fermions lying in the interval $(x,\infty)$. 
Consequently, the negative of its derivative,
\begin{equation*}
    -\frac{\rm d}{{\rm d}x}\frac{1}{\pi}\int_0^{\chi(x)}(D(\theta)-x){\rm d}\theta
    = \frac{\chi(x)}{\pi}
\end{equation*}
naturally identifies with the fermion density function.
Under the $p$-multicriticality condition (iii), we find $D(\theta)= a-\theta^p+O(\theta^{p+1})$. 
This implies that, near the edge $x=a$ (where $a=b$), the density scales proportionally to $(a-x)^{\frac{1}{p}}$.

The right-hand side of the above formula appears in the bulk scaling limit (\ref{eq:diagonalbulk}).
Thus, evaluating the limit
\begin{equation*}
    \lim_{\varepsilon \to 0}\frac{1}{\varepsilon} \rho\left(\frac{a+\varepsilon^{p}x}{\varepsilon^{p+1}};\frac{t}{\varepsilon^{p+1}}\right)
\end{equation*}
should yield precise insights into the microscopic fluctuations of the profile function $\psi_{\lambda,\varepsilon}$ near the edge $x=a$.
In what follows, we determine the more general scaling limit of the $N$-point function,
\begin{equation*}
    \lim_{\varepsilon \to 0}\frac{1}{\varepsilon^N}\rho\left(\left\{\frac{a}{\varepsilon^{p+1}}+\frac{k_i}{\varepsilon}\right\}_{i=1}^N;\frac{t}{\varepsilon^{p+1}}\right)
\end{equation*}
under this $p$-multicritical condition.

Let us define
\begin{equation*}
    J(i;t) = \frac{1}{2\pi{\rm i}} \oint_{|z|=1} \frac{1}{z^{i+1}}
    \mathbb{J}(z;t) {\rm{d}}z.
\end{equation*}
Using \eqref{eq:K}, the correlation kernel can then be expressed as
\begin{equation*}
    K(i,j;t) = \frac{1}{2}J(i;t)J(j;t) + \sum_{k=1}^\infty (-1)^k J(i+k;t)J(j-k;t).
    \label{eq:Kexpansion}
\end{equation*}
Furthermore, by introducing the shifted phase function
\begin{equation*}
    \varphi_{\varepsilon}(\theta)=\varphi(\theta)-\varepsilon^{p}x\theta,
\end{equation*}
the scaled coefficients neatly take the integral form
\begin{equation}
    J\left(\frac{a}{\varepsilon^{p+1}}+\frac{x}{\varepsilon};\frac{t}{\varepsilon^{p+1}}\right)=\frac{1}{2\pi}\int_{-\pi}^{\pi}\exp\left(\frac{\mathrm{i}}{\varepsilon^{p+1}}\varphi_{\varepsilon}(\theta)\right)\mathrm{d}\theta.
    \label{eq:int_for_J}
\end{equation}

\begin{theorem}\label{thm:Airy}
    Assuming conditions (i)--(iv) hold, we have 
    \begin{equation*}
        \lim_{\varepsilon \to 0} \frac{1}{\varepsilon} J\left(\frac{a}{\varepsilon^{p+1}}+\frac{x}{\varepsilon};\frac{t}{\varepsilon^{p+1}}\right) = \mathrm{Ai}_p(x),
    \end{equation*}
    where $\mathrm{Ai}_p$ denotes the $p$-Airy function (see \autoref{app:higher_Airy}).
\end{theorem}

\begin{proof}
Our strategy is to isolate the saddle points of the phase function $\varphi_\varepsilon$ from \eqref{eq:int_for_J} within an interval of width $O(\varepsilon)$. 
The dominant contribution from this central interval will naturally yield the $p$-Airy function.

Let $C$ denote the constant introduced in (\ref{eq:C}). 
We fix a real number $T$ strictly greater than $\sqrt[p]{|Cx|}$, and choose $\varepsilon > 0$ sufficiently small so that $T\varepsilon < \pi$. 
Since the derivative satisfies $\varphi_{\varepsilon}'(\theta)\le-\theta^p/C-\varepsilon^{p}x$, any existing zeros of $\varphi_{\varepsilon}'(\theta)$ must have an absolute value bounded by $\sqrt[p]{|Cx|}\varepsilon$, which is strictly less than $T\varepsilon$. 
By scaling the integration variable as $\theta = \varepsilon u$, we can partition the integral as follows:
\begin{equation*}
    \frac{1}{2\pi\varepsilon} 
    \int_{-\pi}^{\pi}\exp\left(\frac{\mathrm{i}}{\varepsilon^{p+1}}\varphi_{\varepsilon}(\theta)\right)\mathrm{d}\theta 
    = \frac{1}{2\pi}\left(\int_{-\pi/\varepsilon}^{-T}+\int_{-T}^{T}+\int_{T}^{\pi/\varepsilon}\right)\exp\left(\frac{\mathrm{i}}{\varepsilon^{p+1}}\varphi_{\varepsilon}(\varepsilon u)\right)\mathrm{d}u,
\end{equation*}
where, crucially, the derivative $\varphi_{\varepsilon}'(\varepsilon u)$ is non-vanishing on the outer domains $u \in (-\pi/\varepsilon, \pi/\varepsilon) \setminus (-T, T)$. 

As $\varepsilon \to 0$, we observe the asymptotic behaviour of the scaled phase:
\begin{equation*}
    \frac{1}{\varepsilon^{p+1}}\varphi_{\varepsilon}(\varepsilon u)
    =\frac{1}{\varepsilon^{p+1}}\left(-\frac{(\varepsilon u)^{p+1}}{p+1}+O((\varepsilon u)^{p+3})-\varepsilon^{p+1}xu\right)\to-\frac{u^{p+1}}{p+1}-xu.
\end{equation*}
Invoking the dominated convergence theorem for the central interval, we obtain 
\begin{align*}
    \lim_{\varepsilon\to0}\frac{1}{2\pi}\int_{-T}^{T}\exp\left(\frac{\mathrm{i}}{\varepsilon^{p+1}}\varphi_{\varepsilon}(\varepsilon u)\right)\mathrm{d}u 
    &= \frac{1}{2\pi}\int_{-T}^{T}\exp\left(-\mathrm{i}\frac{u^{p+1}}{p+1}-\mathrm{i}xu\right)\mathrm{d}u   \notag\\
    & =\frac{1}{\pi}\int_{0}^{T}\cos\left(\frac{u^{p+1}}{p+1}+xu\right)\mathrm{d}u.
\end{align*}
Conversely, for the outer region $[T, \pi/\varepsilon]$, integration by parts yields 
\begin{align*} 
    &\quad\,\,\int_{T}^{\pi/\varepsilon}\exp\left(\frac{\mathrm{i}}{\varepsilon^{p+1}}\varphi_{\varepsilon}(\varepsilon u)\right)\mathrm{d}u \\
    &\qquad=\left[\frac{\exp\left(\frac{\mathrm{i}}{\varepsilon^{p+1}}\varphi_{\varepsilon}(\varepsilon u)\right)}{\frac{1}{\varepsilon^{p}}\varphi_{\varepsilon}'(\varepsilon u)}\right]_{T}^{\pi/\varepsilon}
    +\int_{T}^{\pi/\varepsilon}\exp\left(\frac{\mathrm{i}}{\varepsilon^{p+1}}\varphi_{\varepsilon}(\varepsilon u)\right)\frac{\frac{\mathrm{i}}{\varepsilon^{p-1}}\varphi_{\varepsilon}''(\varepsilon u)}{\left(\frac{1}{\varepsilon^{p}}\varphi_{\varepsilon}'(\varepsilon u)\right)^{2}}\mathrm{d}u   \notag\\
    &\qquad= O(\varepsilon^{p})+O\left(\frac{1}{T^{p}}\right), \notag
\end{align*}
with an entirely analogous estimate holding for the integral over the negative outer region $[-\pi/\varepsilon, -T]$. 

Taking the limit as $\varepsilon \to 0$, followed by the limit as $T \to \infty$, we conclude that 
\begin{align*}
    \lim_{\varepsilon\to0}\frac{1}{2\pi}\int_{-\pi/\varepsilon}^{\pi/\varepsilon}\exp\left(\frac{\mathrm{i}}{\varepsilon^{p+1}}\varphi_{\varepsilon}(\varepsilon u)\right)\mathrm{d}u 
    &= \frac{1}{\pi}\int_{0}^{T}\cos\left(\frac{u^{p+1}}{p+1}+xu\right)\mathrm{d}u+O\left(\frac{1}{T^{p}}\right)    \notag\\
    &= \mathrm{Ai}_{p}(x).
\end{align*}
\end{proof}

\begin{theorem}
\label{thm:kernel_limit}
Assuming conditions (i)--(iv) hold, for any $x,y \in \mathbb{R}$, we have
\begin{align*}
    \begin{cases}
    \mathrm{(i)} & \quad\lim_{\varepsilon\to0}\frac{1}{\varepsilon}K\left(\frac{a}{\varepsilon^{p+1}}+\frac{x}{\varepsilon},\frac{a}{\varepsilon^{p+1}}+\frac{y}{\varepsilon};\frac{t}{\varepsilon^{p+1}}\right)=0,\\[1ex]
    \mathrm{(ii)} & \quad\lim_{\varepsilon\to0}\frac{1}{\varepsilon}(-1)^{\frac{a}{\varepsilon^{p+1}}+\frac{y}{\varepsilon}}K\left(\frac{a}{\varepsilon^{p+1}}+\frac{x}{\varepsilon},-\frac{a}{\varepsilon^{p+1}}-\frac{y}{\varepsilon};\frac{t}{\varepsilon^{p+1}}\right)=K_{p\text{-}\mathrm{Airy}}(x,y),\\[1ex]
    \mathrm{(iii)} & \quad\lim_{\varepsilon\to0}\frac{1}{\varepsilon}(-1)^{\frac{a}{\varepsilon^{p+1}}+\frac{x}{\varepsilon}+\frac{a}{\varepsilon^{p+1}}+\frac{y}{\varepsilon}}K\left(-\frac{a}{\varepsilon^{p+1}}-\frac{x}{\varepsilon},-\frac{a}{\varepsilon^{p+1}}-\frac{y}{\varepsilon};\frac{t}{\varepsilon^{p+1}}\right)=0.
    \end{cases}
\end{align*}
Consequently, we obtain the determinantal structure in the limit:
\begin{equation*}
\lim_{\varepsilon\to0}\varepsilon^N\mathrm{Pf}\left(M\left(\left\{ \frac{a}{\varepsilon^{p+1}}+\frac{k_{i}}{\varepsilon}\right\}_{i=1}^N ;\frac{t}{\varepsilon^{p+1}}\right)\right)=\det_{1 \le i, j \le N} K_{p\text{-}\mathrm{Airy}}(k_{i},k_{j}).
\end{equation*}
\end{theorem}

\begin{remark}
    The theorem is stated for a pair of numbers $x,y$ such that there exists a sequence of positive numbers $\varepsilon_1>\varepsilon_2>\cdots \to 0$ with $\frac{a}{\varepsilon_i^{p+1}}+\frac{x}{\varepsilon}, \frac{a}{\varepsilon_i^{p+1}}+\frac{y}{\varepsilon} \in \mathbb{Z}$.
    However, the left-hand sides admit double contour integral (with respect to $z,w$) expressions, which do not contain the sign (cf. (\ref{eq:chamu})).
    Then, by dividing the contours into subarcs and fixing branches for $\arg z,\arg w$, we obtain continuations of the left-hand sides for arbitrary $x,y\in \mathbb{R}$ (cf. proof of \autoref{thm:limit_shape}).
    Since the convergence is uniform when $x,y$ range over a compact set, we obtain the same results for arbitrary $x,y$ by incorporating the floor function; for example, \[
    \lim_{\varepsilon\to0}\frac{1}{\varepsilon}(-1)^{\left\lfloor\frac{a}{\varepsilon^{p+1}}+\frac{y}{\varepsilon}\right\rfloor}K\left(\left\lfloor\frac{a}{\varepsilon^{p+1}}+\frac{x}{\varepsilon}\right\rfloor,-\left\lfloor\frac{a}{\varepsilon^{p+1}}+\frac{y}{\varepsilon}\right\rfloor;\frac{t}{\varepsilon^{p+1}}\right)=K_{p\text{-}\mathrm{Airy}}(x,y).
    \]
\end{remark}

\begin{lemma}
\label{lem:J_bound}
There exists a positive constant $C$ such that 
\begin{equation*}
    \left|\frac{1}{\varepsilon}J\left(\frac{a}{\varepsilon^{p+1}}+\frac{x}{\varepsilon};\frac{t}{\varepsilon^{p+1}}\right)\right|\le Ce^{-x},\quad x>0
\end{equation*}
for all $\varepsilon > 0$ sufficiently small. 
\end{lemma}
\begin{proof}
By applying the residue theorem to (\ref{eq:mathbbJ}), we have 
\begin{align*}
    J(m;t) 
    & =\frac{1}{2\pi\mathrm{i}}\oint_{|z|=e^{\varepsilon}}\exp\left(2\sum_{n:\text{odd}}t_n(z^{n}-z^{-n})\right)\frac{\mathrm{d}z}{z^{m+1}}  \notag\\
    & =\frac{1}{2\pi}\int_{-\pi}^{\pi}\exp\left(2\sum_{n:\text{odd}}t_n(e^{n\varepsilon}e^{\mathrm{i}n\theta}-e^{-n\varepsilon}e^{-\mathrm{i}n\theta})\right)e^{-m\varepsilon}e^{-\mathrm{i}m\theta}\mathrm{d}\theta
\end{align*}
yielding 
\begin{equation*}
    \left|J\left(\frac{a}{\varepsilon^{p+1}}+\frac{x}{\varepsilon};\frac{t}{\varepsilon^{p+1}}\right)\right|\le\frac{e^{-x}}{2\pi}\int_{-\pi}^{\pi}\exp\left(\frac{1}{\varepsilon^{p+1}}\left\{ 2\sum_{n:\text{odd}}t_n(e^{n\varepsilon}-e^{-n\varepsilon})\cos n\theta-a\varepsilon\right\} \right)\mathrm{d}\theta.
\end{equation*}
It thus suffices to show that the integral on the right-hand side, multiplied by $1/\varepsilon$, remains bounded as $\varepsilon \to 0$.

Let us observe that 
\begin{align} \label{eq:observation}
    &\quad\,\,2\sum_{n:\text{odd}}t_n(e^{n\varepsilon}-e^{-n\varepsilon})\cos n\theta-a\varepsilon \\
    &=\varepsilon\left\{ 4\sum_{n:\text{odd}}nt_n\cos n\theta-a\right\} +\frac{\varepsilon^{3}}{3!}\left\{ 4\sum_{n:\text{odd}}n^3t_n\cos n\theta\right\} + \cdots \notag\\
    &=\varphi'(\theta)\varepsilon - \varphi^{(3)}(\theta)\frac{\varepsilon^{3}}{3!}+\cdots+(-1)^{\frac{p}{2}+1}\varphi^{(p-1)}(\theta)\frac{\varepsilon^{p-1}}{(p-1)!}+O(\varepsilon^{p+1}). \notag
\end{align}
By virtue of (\ref{eq:C}), for sufficiently small $\varepsilon>0$, we obtain
\begin{equation*}
    2\sum_{n:\text{odd}}t_n(e^{n\varepsilon}-e^{-n\varepsilon})\cos n\theta-a\varepsilon \le -\frac{\varepsilon\theta^p}{2C}.
\end{equation*}
Introducing the scaled variable $\theta = \varepsilon u$, we find 
\begin{align*}
    &\quad\,\,\frac{1}{\varepsilon}\int_{-\pi}^{\pi}\exp\left(\frac{1}{\varepsilon^{p+1}}\left\{ 2\sum_{n:\text{odd}}t_n(e^{n\varepsilon}-e^{-n\varepsilon})\cos n\theta-a\varepsilon\right\} \right)\mathrm{d}\theta   \notag\\
    &\le \int_{-\pi/\varepsilon}^{\pi/\varepsilon}\exp\left(-\frac{u^p}{2C}\right)\mathrm{d}u  \notag\\
    &= O(1) \quad (\text{as }\varepsilon \to 0).
\end{align*}
This bounds the integral and completes the proof.
\end{proof}
\begin{proof}[Proof of Theorem \ref{thm:kernel_limit}]
Rescaling the variables in equation (\ref{eq:K}) according to (ii) and employing the parity relation $J(-i)=(-1)^iJ(i)$, we write
\begin{align*}
    &\quad\,\,\frac{1}{\varepsilon}(-1)^{\frac{a}{\varepsilon^{p+1}}+\frac{y}{\varepsilon}}K\left(\frac{a}{\varepsilon^{p+1}}+\frac{x}{\varepsilon},-\frac{a}{\varepsilon^{p+1}}-\frac{y}{\varepsilon};\frac{t}{\varepsilon^{p+1}}\right) \\
    &= \frac{\varepsilon}{2}\left(\frac{1}{\varepsilon}J\left(\frac{a}{\varepsilon^{p+1}}+\frac{x}{\varepsilon};\frac{t}{\varepsilon^{p+1}}\right)\right)\left(\frac{1}{\varepsilon}J\left(\frac{a}{\varepsilon^{p+1}}+\frac{y}{\varepsilon};\frac{t}{\varepsilon^{p+1}}\right)\right) \\
    &\quad+\sum_{i=1}^{\infty}\varepsilon\left(\frac{1}{\varepsilon}J\left(\frac{a}{\varepsilon^{p+1}}+\frac{1}{\varepsilon}(x+i\varepsilon);\frac{t}{\varepsilon^{p+1}}\right)\right)\left(\frac{1}{\varepsilon}J\left(\frac{a}{\varepsilon^{p+1}}+\frac{1}{\varepsilon}(y+i\varepsilon);\frac{t}{\varepsilon^{p+1}}\right)\right).
\end{align*}

This infinite sum may be interpreted as the Riemann sum of a step function. 
By Lemma \ref{lem:J_bound}, for sufficiently small $\varepsilon$, this step function is uniformly bounded by an exponentially decaying profile. 
Furthermore, by Theorem \ref{thm:Airy}, the step function converges pointwise to the product of $p$-Airy functions.
Invoking Lebesgue's dominated convergence theorem, we thus deduce that 
\begin{align*}
    &\lim_{\varepsilon\to0}\frac{1}{\varepsilon}(-1)^{\frac{a}{\varepsilon^{p+1}}+\frac{y}{\varepsilon}}K\left(\frac{a}{\varepsilon^{p+1}}+\frac{x}{\varepsilon},-\frac{a}{\varepsilon^{p+1}}-\frac{y}{\varepsilon};\frac{t}{\varepsilon^{p+1}}\right)    \notag\\
    &\qquad=\int_0^\infty {\rm Ai}_p(x+z){\rm Ai}_p(y+z){\rm d}z.
\end{align*}

It remains to establish (i) and (iii). We shall detail the proof of (i) alone, as the argument for (iii) is entirely analogous. From equation (\ref{eq:K}), we have 
\begin{align*}
&\frac{1}{\varepsilon}K\left(\frac{a}{\varepsilon^{p+1}}+\frac{x}{\varepsilon},\frac{a}{\varepsilon^{p+1}}+\frac{y}{\varepsilon};\frac{t}{\varepsilon^{p+1}}\right) \\
&= \frac{1}{2(2\pi{\rm i})^2\varepsilon}\oiint_{|z|=e^{2\varepsilon},|w|=e^{\varepsilon}}\exp\left\{\frac{1}{\varepsilon^{p+1}}(S(z,a)+S(w,a))+\frac{1}{\varepsilon}(x\log z+y\log w)\right\}\frac{z-w}{z+w}\frac{{\rm d}z{\rm d}w}{zw}.
\end{align*}
Let us observe that, by (\ref{eq:observation}), 
\begin{equation*}
    \operatorname{Re}S(e^{\varepsilon+{\rm i}\theta},a)
    = \varphi'(\theta)\varepsilon - \varphi^{(3)}(\theta)\frac{\varepsilon^{3}}{3!}+\cdots+(-1)^{\frac{p}{2}+1}\varphi^{(p-1)}(\theta)\frac{\varepsilon^{p-1}}{(p-1)!}+O(\varepsilon^{p+1}).
\end{equation*}
We fix a constant $\delta \in (0,1)$ such that $p\delta>p-2$, and define the \emph{central region} $I$ as 
\begin{equation*}
    I=\{(z,w)\mid \abs{z}=e^{2\varepsilon},\abs{w}=e^\varepsilon,\abs{\arg z},\abs{\arg w}<\varepsilon^{1-\delta}\}.
\end{equation*}
We denote its complement on the integration contours by $I^c=\{\abs{z}=e^{2\varepsilon}\}\times\{\abs{w}=e^\varepsilon\}\setminus I$.

\case{1} $(z,w)\in I^c$. Without loss of generality, let us assume $\abs{\arg z}>\varepsilon^{1-\delta}$. Under this assumption, equation (\ref{eq:C}) yields 
\begin{equation*}
    \frac{1}{\varepsilon^{p+1}}\operatorname{Re}S(z,a)\le \frac{\varphi'(\theta)}{\varepsilon^p}\le-\frac{\theta^p}{C\varepsilon^p}\le-\frac{\varepsilon^{-p\delta}}{C}.
\end{equation*}
Meanwhile, since 
\begin{equation*}
    \frac{{\rm d}}{{\rm d}\varepsilon}\operatorname{Re}S(e^{\varepsilon+{\rm i}\theta},a)=\varphi'(\theta)+O(\varepsilon^2)\le O(\varepsilon^2),
\end{equation*}
the mean value theorem provides 
\begin{equation*}
    \frac{1}{\varepsilon^{p+1}}\operatorname{Re}S(w,a) = O\left(\frac{1}{\varepsilon^{p-2}}\right)=o(\varepsilon^{-p\delta}).
\end{equation*}
Furthermore, we note that 
\begin{equation*}
    \frac{1}{\varepsilon}(x\log z+y\log w) = O(1),\quad \frac{z-w}{z+w}=O\left(\frac{1}{\varepsilon}\right).
\end{equation*}
Consequently, we deduce 
\begin{align*}
    &\abs{\frac{1}{\varepsilon}\oiint_{(z,w)\in I^c}\exp\left\{\frac{1}{\varepsilon^{p+1}}(S(z,a)+S(w,a))+\frac{1}{\varepsilon}(x\log z+y\log w)\right\}\frac{z-w}{z+w}\frac{{\rm d}z{\rm d}w}{zw}}\\
    &\qquad =O\left(\frac{1}{\varepsilon^2}\exp\left(-\frac{\varepsilon^{-p\delta}}{C}\right)\right)\to 0.
\end{align*}

\case{2} $(z,w)\in I$. In this regime, bounding the terms as 
\begin{equation}
    \frac{1}{\varepsilon}\operatorname{Re}S(e^{\varepsilon+{\rm i}\theta},a) \le -\frac{\theta^p}{C\varepsilon^p},
    \quad \frac{1}{\varepsilon}(x\log z+y\log w) = O(1),
    \quad \frac{z-w}{z+w}=O(1),
\label{eq:80}
\end{equation}
yields
\begin{align*}
    &\quad\,\,\abs{\frac{1}{\varepsilon}\oiint_{(z,w)\in I}\exp\left\{\frac{1}{\varepsilon^{p+1}}(S(z,a)+S(w,a))+\frac{1}{\varepsilon}(x\log z+y\log w)\right\}\frac{z-w}{z+w}\frac{{\rm d}z{\rm d}w}{zw}}\\
    &\le \frac{1}{\varepsilon}\int_{-\varepsilon^{1-\delta}}^{\varepsilon^{1-\delta}}\exp\left(-\frac{\theta^p}{C(2\varepsilon)^p}+O(1)\right){\rm d}\theta\cdot
    \int_{-\varepsilon^{1-\delta}}^{\varepsilon^{1-\delta}}
    \exp\left(-\frac{\phi^p}{C\varepsilon^p}+O(1)\right){\rm d}\phi\\
    &=\varepsilon\int_{-\varepsilon^{-\delta}}^{\varepsilon^{-\delta}}\exp\left(-\frac{u^p}{2^pC}+O(1)\right){\rm d}u\cdot
    \int_{-\varepsilon^{-\delta}}^{\varepsilon^{-\delta}}\exp\left(-\frac{v^p}{C}+O(1)\right){\rm d}v\\
    &\to 0,
\end{align*}
where we have introduced the scaled variables $\theta=\varepsilon u$ and $\phi=\varepsilon v$.
\end{proof}

\begin{remark}
The line of reasoning employed in the proof of (i) breaks down for case (ii). 
This fundamentally occurs because the bounded estimate $\frac{z-w}{z+w}=O(1)$ in (\ref{eq:80}) is replaced by the singular behaviour $\frac{z+w}{z-w}=O(\varepsilon^{-1})$.
\end{remark}

%%先生コメント向け
%To conclude this section, let us briefly touch upon the gap probability $P(I)$ for a given subset $I \subset \mathbb{Z}_{\ge 0}$. 
%$P(I)$ can be formulated as a Fredholm Pfaffian. Crucially, as we pass to the multicritical limit, the off-diagonal block entries of the underlying matrix kernel vanish. 
%Consequently, this Fredholm Pfaffian reduces to the Fredholm determinant of the $p$-Airy kernel. 
%We defer a more comprehensive exposition of this mechanism to \autoref{sec:gap_prob}.

%\todo[inline]{Maybe we may move the discussion of \autoref{sec:gap_prob} here and just state the proposition (without giving a proof if we don't have time; one can add a proof in the revised version).}

\subsection*{Gap probability}
To conclude this section, let us briefly touch upon the gap probability.
Let $I \subset \mathbb{Z}_{>0}$ be a discrete interval. 
We consider the gap probability with respect to $I$, such that no point is found in $I$,
\[
P(I) = \sum_{n=0}^\infty \frac{(-1)^n}{n!} \sum_{x_1 \in I} \cdots \sum_{x_n \in I} \rho(x_1,\ldots,x_n;t) \, .
\]
In the case of the shifted Schur measure, the Pfaffian formula of the correlation function (Theorem~\ref{thm:PfPP}) yields the Fredholm Pfaffian formula for the gap probability.
Let $K$ be a linear operator acting on $\ell^2(I;\mathbb{R}^2)$,
\[
K : 
\begin{pmatrix}
    f(x) \\ g(x)
\end{pmatrix} \longmapsto \sum_{y \in I} 
\begin{pmatrix}
    K(x,y) & (-1)^y K(x,-y) \\ (-1)^x K(-x,y) & (-1)^{x+y} K(-x,-y)
\end{pmatrix}
\begin{pmatrix}
    f(y) \\ g(y)
\end{pmatrix} \, ,
\]
and also
\[
J : 
\begin{pmatrix}
    f(x) \\ g(x)
\end{pmatrix} \longmapsto \sum_{y \in I} 
\begin{pmatrix}
    0 & \delta_{x,y} \\ -\delta_{x,y} & 0
\end{pmatrix}
\begin{pmatrix}
    f(y) \\ g(y)
\end{pmatrix} = 
\begin{pmatrix}
    g(x) \\ -f(x)
\end{pmatrix} \, .
\]
Then, the gap probability is given by the Fredholm Pfaffian,
\[
P(I) = \sum_{n=0}^\infty \frac{(-1)^n}{n!} \sum_{x_1 \in I} \cdots \sum_{x_n \in I} \operatorname*{Pf}_{1 \le i, j \le 2n} M(\{x_1,\ldots,x_n\})_{i,j} = \operatorname{Pf}(J - K)_{\ell^2(I;\mathbb{R}^2)} \, .
\]
Arguing the same way as in \cite[Theorem 1]{MR4693927}, we obtain the following result.
\begin{proposition}
    Let $p \in 2\mathbb{Z}_+$. Under the multicritical condition of degree $p$, we have
    \[
    \lim_{\varepsilon \to 0} \mathbb{P} \left[ \frac{\lambda_1 - a/\varepsilon^{p+1}}{\varepsilon} \le s \right] = F_p(s) \, ,
    \]
    where $F_p$ is the degree-$p$ Tracy--Widom distribution,
    \begin{align*}
        F_p(s) & = \det(1 - K_{p\text{-Airy}})_{L^2([s,\infty))} %\nonumber \\ &
        = \sum_{n=0}^\infty \frac{(-1)^n}{n!} \int_s^\infty \cdots \int_s^\infty \det_{1 \le i, j \le n} K_{p\text{-Airy}}(x_i,x_j) \, \mathrm{d}x_1 \cdots \mathrm{d}x_n \, .
    \end{align*}
\end{proposition}

The simplest case $p = 2$ was proved by Matsumoto~\cite{MR2139724}, where $F_2(s)$ agrees with the GUE Tracy--Widom distribution~\cite{MR1257246}.

\appendix
\section{Higher Airy Functions}
\label{app:higher_Airy}

For any even integer $p \ge 2$, the \emph{$p$-Airy function} $\mathrm{Ai}_{p}(x)$ is defined by the integral
\begin{equation*}
    \mathrm{Ai}_{p}(x) = \frac{1}{\pi} \int_{0}^{\infty} \cos \left( \frac{t^{p+1}}{p+1} + xt \right) \mathrm{d}t, \quad x \in \mathbb{R}.
\end{equation*}
The convergence of this integral is verified via integration by parts:
\begin{align*}
\int^{R} \cos \left( \frac{t^{p+1}}{p+1} + xt \right) \mathrm{d}t 
  &= \int^{R} \frac{1}{t^{p}+x} \mathrm{d} \left( \sin \left( \frac{t^{p+1}}{p+1} + xt \right) \right) \notag\\
  &= \left[ \frac{\sin \left( \frac{t^{p+1}}{p+1} + xt \right)}{t^{p}+x} \right]^{R} + \int^{R} \frac{p t^{p-1}}{(t^{p}+x)^{2}} \sin \left( \frac{t^{p+1}}{p+1} + xt \right) \mathrm{d}t.
\end{align*}

To obtain the analytic continuation of $\mathrm{Ai}_{p}$ into the entire complex plane, we substitute $\zeta = \mathrm{i}t$, which yields
\begin{equation}
    \mathrm{Ai}_{p}(x) = \frac{1}{2\pi\mathrm{i}} \int_{-\mathrm{i}\infty}^{\mathrm{i}\infty} \exp \left( (-1)^{\frac{p}{2}+1} \frac{\zeta^{p+1}}{p+1} - x\zeta \right) \mathrm{d}\zeta.
\label{eq:app_int_rep}
\end{equation}
Consider the integral
\begin{equation*}
    I(R) = \int_{\gamma} \left| \exp \left( (-1)^{\frac{p}{2}+1} \frac{\zeta^{p+1}}{p+1} - x\zeta \right) \right| |\mathrm{d}\zeta|,
\end{equation*}
where the contour $\gamma$ is parameterised by $\gamma(\theta) = \mathrm{i} R e^{-\mathrm{i}\frac{\theta}{p+1}}$ for $0 \le \theta \le \pi$.

Applying Jordan's inequality, we obtain the estimate
\begin{align*}
    I(R) 
    &= \frac{R}{p+1} \int_{0}^{\pi} \exp \left( -\frac{R^{p+1}}{p+1} \sin\theta - xR \sin \frac{\theta}{p+1} \right) \mathrm{d}\theta \notag\\
    &\le \frac{\pi R}{p+1} \exp \left( -\frac{R^{p+1}}{p+1} \frac{2}{\pi} + |x|R \right).
\end{align*}
It follows that $I(R)$ vanishes as $R \to \infty$. 
Consequently, the integration contour in (\ref{eq:app_int_rep}) may be deformed into any path originating at infinity within the sector $-\frac{\pi}{2} \le \arg \zeta \le -\frac{\pi}{2} + \frac{\pi}{p+1}$ and terminating at infinity in the conjugate sector $\frac{\pi}{2} - \frac{\pi}{p+1} \le \arg \zeta \le \frac{\pi}{2}$.

For a contour $\mathscr{L}$ that originates in the sector $-\frac{\pi}{2} + \delta \le \arg \zeta \le -\frac{\pi}{2} + \frac{\pi}{p+1} - \delta$ and terminates in $\frac{\pi}{2} - \frac{\pi}{p+1} + \delta \le \arg \zeta \le \frac{\pi}{2} - \delta$ for some $\delta > 0$, we define
\begin{equation}
    \mathrm{Ai}_{p}(z) = \frac{1}{2\pi\mathrm{i}} \int_{\mathscr{L}} \exp \left( (-1)^{\frac{p}{2}+1} \frac{\zeta^{p+1}}{p+1} - z\zeta \right) \mathrm{d}\zeta, \quad z \in \mathbb{C}.
\label{eq:app_Ai_complex}
\end{equation}
At the extremities of $\mathscr{L}$, the term $\exp(\zeta^{p+1}/(p+1))$ dominates $\exp(-z\zeta)$, ensuring that the integral (\ref{eq:app_Ai_complex}) converges absolutely and uniformly on any compact subset of $\mathbb{C}$. This provides the analytic continuation of $\mathrm{Ai}_{p}$ as an entire function. Henceforth, $\mathrm{Ai}_{p}$ shall refer to this holomorphic function.

By differentiating under the integral sign, it is evident that $\mathrm{Ai}_{p}$ satisfies the \emph{$p$-Airy equation}
\begin{equation*}
   \left[ (-1)^{\frac{p}{2}+1} \frac{\mathrm{d}^{p}}{\mathrm{d}z^{p}} - z \right] u = 0. 
\end{equation*}
This equation possesses $p$ linearly independent solutions. For each root $\tau$ of $\tau^{p} = (-1)^{\frac{p}{2}+1}$, there exists a solution with the asymptotic growth rate $\log u(x) \sim -\frac{p}{p+1} (\operatorname{Re} \tau) x^{\frac{p+1}{p}}$ as $\mathbb{R} \ni x \to \infty$ (see, e.g., \cite{MR203188}). A standard application of the saddle point method shows that
\begin{equation}
    \log \mathrm{Ai}_{p}(x) \sim -\frac{p}{p+1} (\operatorname{Re} \tau_{*}) x^{\frac{p+1}{p}} \quad \text{as } x \to \infty,
\label{eq:app_asymptotic}
\end{equation}
where $\operatorname{Re}\tau_{*} = \max \{ \operatorname{Re} \tau \mid \tau^{p} = (-1)^{\frac{p}{2}+1} \}$. 
Accordingly, $\mathrm{Ai}_{p}$ is the fastest decaying real-valued solution of the $p$-Airy equation as $x \to \infty$.

By virtue of the asymptotic estimate (\ref{eq:app_asymptotic}), we define the \emph{$p$-Airy kernel} as
\begin{equation*}
    K_{p\text{-}\mathrm{Airy}}(z,w) = \int_{0}^{\infty} \mathrm{Ai}_{p}(z+x) \mathrm{Ai}_{p}(w+x) \mathrm{d}x.
\end{equation*}

% References

\bibliographystyle{ytamsalpha}
\bibliography{ref}

\end{document}